\documentclass[letterpaper, 10 pt, conference]{ieeeconf}  % Comment this line out if you need a4paper

\IEEEoverridecommandlockouts                              % This command is only needed if
\usepackage{mathtools}
\usepackage{amsmath,amssymb}%amsthm
\usepackage{graphicx}
\usepackage{epstopdf,cite}
\usepackage{url}
\usepackage{bm}
\usepackage{multirow,booktabs,multicol}
\usepackage{mathtools}
\usepackage{setspace}
\usepackage[colorlinks = true,
            linkcolor = blue,
            urlcolor  = blue,
            citecolor = blue,
            anchorcolor = blue]{hyperref}
\usepackage{makecell}
\usepackage{pifont}% http://ctan.org/pkg/pifont
\usepackage{graphicx,color}
\usepackage{subfigure,balance,stfloats}
\usepackage[export]{adjustbox}
\usepackage[english]{babel}

\usepackage{tikz}

\usepackage[font=small,labelfont=bf,   justification=justified]{caption}
\newcommand{\tr}{{{\mathsf T}}}
\newcommand{\mK}{{\mathsf{K}}}

\newcommand{\blue}[1]{{\color{blue} #1}}
\newtheorem{definition}{Definition}
\newtheorem{theorem}{Theorem}

\newtheorem{remark}{Remark}
\newtheorem{lemma}{Lemma}

\newtheorem{corollary}{Corollary}
\newtheorem{assumption}{Assumption}

\usepackage{tikz}
\usetikzlibrary{decorations.pathreplacing,decorations.shapes}
\usetikzlibrary{shapes,arrows,chains,matrix}

\usepackage[nameinlink]{cleveref}
\crefname{equation}{}{}
\crefname{theorem}{Theorem}{Theorems}
\crefname{corollary}{Corollary}{Corollaries}
\crefname{example}{Example}{Examples}
\crefname{assumption}{Assumption}{Assumptions}
\crefname{lemma}{Lemma}{Lemmas}
\crefname{proposition}{Proposition}{Propositions}
\crefname{figure}{Figure}{Figures}
\crefname{table}{Table}{Tables}
\crefname{section}{Section}{Sections}
\crefname{appendix}{Appendix}{Appendices}
\Crefname{equation}{}{}
\Crefname{theorem}{Theorem}{Theorems}
\Crefname{corollary}{Corollary}{Corollaries}
\Crefname{example}{Example}{Examples}
\Crefname{lemma}{Lemma}{Lemma}
\Crefname{proposition}{Proposition}{Proposition}
\Crefname{figure}{Figure}{Figures}
\Crefname{table}{Table}{Tables}
\Crefname{section}{Section}{Sections}
\Crefname{appendix}{Appendix}{Appendices}

\title{\LARGE \bf
Convex Parameterization of Stabilizing Controllers and its LMI-based Computation via Filtering
}

\author{Mauricio C. de Oliveira$^{1}$ and Yang Zheng$^{2}$% <-this % stops a space
\thanks{$^{1}$M. C. de Oliveira is with the Department of Mechanical and Aerospace Engineering, University of California San Diego, CA 92093, USA. ({mauricio@ucsd.edu})}%
\thanks{$^{2}$Y. Zheng is with the Department of Electrical and Computer Engineering, University of California San Diego, CA 92093, USA. ({zhengy@eng.ucsd.edu})}%
}

\begin{document}

\maketitle
\thispagestyle{empty}
\pagestyle{empty}

%%%%%%%%%%%%%%%%%%%%%%%%%%%%%%%%%%%%%%%%%%%%%%%%%%%%%%%%%%%%%%%%%%%%%%%%%%%%%%%%
\begin{abstract}

Various new implicit parameterizations for stabilizing controllers that allow one to impose structural constraints on the controller have been proposed lately. They are convex but infinite-dimensional, formulated in the frequency~domain with no available efficient methods for computation. In this paper, we introduce a kernel version of the Youla parameterization to characterize the set of stabilizing controllers. It features a single affine constraint, which allows us to recast the controller parameterization as a novel robust filtering problem. This~makes it possible to derive %what we believe to be 
the \textit{first} efficient Linear Matrix Inequality (LMI) implicit parametrization of stabilizing controllers. Our LMI characterization not only admits efficient numerical computation, but also guarantees a full-order stabilizing dynamical controller that is efficient for practical deployment. Numerical experiments demonstrate that our LMI can be orders of magnitude faster to solve than the existing closed-loop parameterizations.
\end{abstract}

%%%%%%%%%%%%%%%%%%%%%%%%%%%%%%%%%%%%%%%%%%%%%%%%%%%%%%%%%%%%%%%%%%%%%%%%%%%%%%%%
\section{Introduction}

One basic yet fundamental problem in control theory is that of designing a feedback controller to stabilize a dynamical system~\cite[Chapter 12]{zhou1996robust}. Any controller synthesis method needs to implicitly or explicitly include stability as a constraint, since feedback systems must be stable for practical deployment. When the system state is directly measured, it is sufficient to consider a static state feedback $u = Kx$ with a constant matrix $K$. In this case, the set of stabilizing gains can be characterized by a Lyapunov inequality. If we only have  output measurements, a static output feedback is insufficient to get good closed-loop performance. Instead, we need to consider the class of dynamical controllers~\cite{zhou1996robust,boyd1991linear,francis1987course}. 

It is well-known that the set of stabilizing~dynamical~controllers is characterized by the classical \textit{Youla parameterization}~\cite{youla1976modern} in the frequency domain, which requires  a \textit{doubly coprime factorization} of the system. Many closed-loop performances can be further addressed via convex optimization in the Youla framework; see~\cite{boyd1991linear} for extensive discussions.~In the past few years, a classical notion of \textit{closed-loop convexity} (coined in~\cite[Chapter 6]{boyd1991linear}) has regained increasing~attention thanks to its flexibility in addressing distributed control and data-driven control problems~\cite{anderson2019system,furieri2019input,zheng2021sample,dean2020sample,zhang2021sample,zheng2019equivalence,tseng2021realization,zheng2019systemlevel,wang2019system,furieri2019sparsity,rotkowitz2005characterization}. One common underlying idea is to parameterize stabilizing dynamical controllers using certain closed-loop responses in a convex way, which shifts %the controller synthesis task 
from designing a controller to designing desirable closed-loop responses. One main benefit is that designing~closed-loop responses becomes a convex problem in many distributed and data-driven control setups~\cite{rotkowitz2005characterization,wang2019system,furieri2019sparsity}.  

In particular, a system-level parameterization (SLP) was introduced in~\cite{wang2019system}, and an input-output parameterization (IOP) was proposed in~\cite{furieri2019input}; both of them characterize the set of all stabilizing dynamical controllers with no need of computing a doubly-coprime factorization explicitly. As expected, Youla, SLP, and IOP are equivalent to each other in theory, which has been first proved in~\cite{zheng2019equivalence} and later discussed in~\cite{tseng2021realization}. Very recently, the work~\cite{zheng2019systemlevel} has further characterized all convex parameterizations of stabilizing controllers using closed-loop responses, revealing two new parameterizations beyond SLP and IOP. Thanks to convexity, these closed-loop parameterizations have become powerful tools in addressing various distributed control problems~\cite{anderson2019system,furieri2019sparsity}, and quantifying the performance of data-driven control~\cite{dean2020sample,zheng2021sample,zhang2021sample}. 

While convexity is one desirable feature in closed-loop parameterizations, the resulting convex problems are unfortunately always \textit{infinitely dimensional} since the decision variables are transfer functions in the frequency domain. The classical work~\cite{boyd1991linear} and all the recent studies~\cite{anderson2019system,wang2019system,furieri2019input,zheng2021sample,dean2020sample,zhang2021sample,zheng2019equivalence,tseng2021realization,zheng2019systemlevel,furieri2019sparsity} apply Ritz or finite impulse response (FIR) approximations for numerical computation. However, the Ritz or FIR approximations do not scale well in both computational efficiency and controller implementation since they lead to large-scale optimization problems and result in dynamical controllers of impractical high-order. Moreover, a subtle notion~of~\textit{numerical robustness}~\cite[Section 6]{zheng2019systemlevel} arises on the SLP~\cite{wang2019system} and IOP~\cite{furieri2019input} due to the FIR approximation that may affect internal stability in practical computation.   

In this paper, we present the \textit{first} computationally efficient linear matrix inequality (LMI) characterization for a closed-loop parameterization of stabilizing dynamical controllers. To achieve this,  we first introduce a ``kernel'' version of the Youla parameterization. Unlike SLP~\cite{wang2019system}, IOP~\cite{furieri2019input} and the mixed parameterizations~\cite{zheng2019systemlevel}, our new parameterization only requires one single affine constraint. This feature leads to a new robust $\mathcal{H}_\infty$ filtering problem, which  allows us to derive an LMI for efficient computation. Note that our filtering problem is different from the classical setup (cf.~\cite{geromel2000h,geromel2002robust}), and thus our LMI characterization might have independent interest. Numerical experiments show that our LMI can be orders of magnitude faster to solve than FIR approximations. % in SLP. 

The rest of this paper is organized as follows. We present the problem statement in \cref{section:problem-statement}. Our new parameterization is presented in \Cref{section:parameterization}, and its LMI characterization  is introduced in \Cref{section:filtering}. Numerical results are shown in \Cref{section:experiments}. We conclude the paper in \Cref{section:conclusion}. %Some proofs are in the appendix. 

%\vspace{1pt}
% \noindent\emph{Notation:} Most notation is standard in this paper. We use lower and upper case letters (\emph{e.g.} $x$ and $A$) to denote vectors and matrices, respectively. We use lower and upper case boldface letters (\emph{e.g.} $\mathbf{x}$ and $\mathbf{G}$) to denote signals and transfer matrices, respectively. We denote the set of real-rational proper stable transfer matrices as $\mathcal{RH}_{\infty}$. 
% % The state-space realization $C(zI - A)^{-1}B + D$ is denoted as
% %  $ \left[\begin{array}{c|c} A & B \\\hline
% %      C & D\end{array}\right].$

%\clearpage 

\section{Preliminaries and Problem Statement} \label{section:problem-statement}

\subsection{System model and internal stability}

We consider a strictly proper linear time-invariant (LTI) plant in the discrete-time domain\footnote{Unless specified otherwise, all the results in this paper can be generalized to continuous-time systems.}
\begin{equation}\label{eq:LTI}
    \begin{aligned}
        x[t+1] &= A x[t] + B u[t] + \delta_x[t],\\
        y[t]   &= Cx[t]  + \delta_y[t],
    \end{aligned}
\end{equation}
where $x[t] \in \mathbb{R}^{n},u[t]\in \mathbb{R}^{m},y[t]\in \mathbb{R}^{p}$ are the state, control action, and measurement vector at time $t$, respectively, and $\delta_x[t]\in \mathbb{R}^{n}$ and $\delta_y[t]\in \mathbb{R}^{p}$ are disturbances on the state and measurement vectors at time $t$, respectively. The transfer matrix from $\mathbf{u}$ to $\mathbf{y}$ is
$
    \mathbf{G} = C(zI - A)^{-1}B,
$
where $z \in \mathbb{C}$. 

Consider an output-feedback LTI dynamical controller
\begin{equation}\label{eq:LTIController}
    \mathbf{u} = \mathbf{K}\mathbf{y} + \bm{\delta}_u,
\end{equation}
where $\bm{\delta}_u $ is the external disturbance on the control action. The controller~\cref{eq:LTIController} has a state-space realization as
\begin{equation}\label{eq:state_space_controller}
    \begin{aligned}
        \xi[t+1] &= A_\mK \xi[t] + B_\mK y[t],\\
        u[t]   &= C_\mK\xi[t] + D_\mK y[t]+ \delta_u[t],
    \end{aligned}
\end{equation}
where $\xi[t]\in \mathbb{R}^{q}$ is the controller internal state at time $t$,  and $A_\mK \in \mathbb{R}^{q \times q}, B_\mK \in \mathbb{R}^{q \times p}, C_\mK \in \mathbb{R}^{m \times q}, D_\mK \in \mathbb{R}^{m \times p}$ specify the controller dynamics. We call $q$ the order of the controller $\mathbf{K}$.  %{The formulation~\cref{eq:state_space_controller} above reduces to a static controller when $(A_\mK,B_\mK,C_\mK,D_\mK)=(0,0,0,K)$ for some $K \in \mathbb{R}^{m \times p}$. }
Applying the controller~\cref{eq:LTIController} to the plant~\cref{eq:LTI} leads to a closed-loop system shown in \Cref{Fig:LTI}. We make the following standard assumption.
%\vspace{-4mm}
\begin{assumption} \label{assumption:stabilizability}
    The plant is stabilizable and detectable, \emph{i.e.}, $(A, B)$ is stabilizable, and $(C,A)$ is detectable.
\end{assumption}

    The closed-loop system must be stable in some appropriate sense, and any controller synthesis procedure implicitly or explicitly involves a stability constraint~\cite{zhou1996robust,dullerud2013course,boyd1991linear,youla1976modern,francis1987course,wang2019system,furieri2019input,zheng2019equivalence}. A standard~notion is \textit{internal stability}, defined as~\cite[Chapter 5.3]{zhou1996robust}:
\begin{definition}
    The system in \Cref{Fig:LTI} is \emph{internally stable} if it is well-posed, and the states $(x[t],\xi[t])$ converge to zero as $t\rightarrow \infty$ for all initial states $(x[0],\xi[0])$ when $\delta_x[t] = 0, \delta_y[t]=0, \delta_u[t] = 0, \forall t$.
\end{definition}

  The interconnection in \Cref{Fig:LTI} is always well-posed since the plant is strictly proper~\cite[Lemma 5.1]{zhou1996robust}. 
We say the controller $\mathbf{K}$ \emph{internally stabilizes} the plant $\mathbf{G}$ if the closed-loop system in \Cref{Fig:LTI} is {internally stable}. The set of all internally stabilizing LTI dynamical controllers is defined as
\begin{equation}
    \mathcal{C}_{\text{stab}} := \{\mathbf{K} \mid \mathbf{K} \; \text{internally stabilizes} \; \mathbf{G}\}.
\end{equation}
\begin{figure}[t]
\setlength{\belowcaptionskip}{-10pt}
\setlength{\abovecaptionskip}{-10pt}
  \centering
  \begin{center}
  \setlength{\unitlength}{0.008in}
\begin{picture}(243,170)(140,410)
\thicklines

%LOWER-LEFT Corner
\put(147,440){\vector(1, 0){ 93}}
\put(141,540){\vector( 0, -1){ 95}}
\put(88,440){\vector( 1, 0){ 48}}
\put(141,440){\circle{10}}
%\put(180,510){\line( 1, 0){ 48.5}}
\put(170,445){\makebox(0,0)[lb]{$\mathbf{y}$}}
\put(100,445){\makebox(0,0)[lb]{$\delta_{\mathbf{y}}$}}
\put(145,452){\makebox(0,0)[lb]{\small +}}
\put(123,446){\makebox(0,0)[lb]{\small +}}

%UPPER-RIGHT Corner

\put(385,540){\circle{10}}
\put(380,540){\vector( -1, 0){ 25}}
\put(385,439.5){\vector( 0,1){ 95.5}}
\put(438,540){\vector(-1, 0){ 48}}
\put(410,545){\makebox(0,0)[lb]{$\delta_{\mathbf{u}}$}}
\put(365,545){\makebox(0,0)[lb]{$\mathbf{u}$}}
\put(220,545){\makebox(0,0)[lb]{$\mathbf{x}$}}
\put(398,545){\makebox(0,0)[lb]{\small +}}
\put(390,524){\makebox(0,0)[lb]{\small +}}
\put(266,440){\line(1, 0){ 119.5}}

\put(240,427){\framebox(25,25){}} %K box
\put(245,435){\makebox(0,0)[lb]{$\mathbf{K}$}}

\put(330,527){\framebox(25,25){}} %B box
\put(335,535){\makebox(0,0)[lb]{$B$}}

\put(312,545){\makebox(0,0)[lb]{\small +}}
\put(302,550){\makebox(0,0)[lb]{\small +}}
\put(302,525){\makebox(0,0)[lb]{\small +}}

%%INNER CICLE
\put(329,540){\vector(-1, 0){26}}
\put(298,540){\circle{10}}
\put(298,580){\vector(0, -1){35}}
\put(301,570){\makebox(0,0)[lb]{$\delta_{\mathbf{x}}$}}
\put(293,540){\vector(-1, 0){25}}

\put(243,527){\framebox(25,25){}} %z^-1 box
\put(246,534){\makebox(0,0)[lb]{\small  $z^{\text{-}1}$}}

\put(242,540){\vector(-1, 0){46}}
\put(219,540){\line(0, -1){46}}
\put(218.5,494){\vector(1, 0){24}}

\put(243,482){\framebox(25,25){}} %z^-1 box
\put(248,490){\makebox(0,0)[lb]{$A$}}

\put(268.5,494){\line(1, 0){29.7}}
\put(298.5,493.5){\vector(0, 1){42}}

%C Part
\put(171,527){\framebox(25,25){}} %z^-1 box
\put(176,535){\makebox(0,0)[lb]{$C$}}
\put(170.5,540){\line(-1, 0){30}}

%RED BOX?

{\color{red}
\put(160,475){\dashbox(203,90){}} %B box
\put(339,480){\makebox(0,0)[lb]{$\mathbf{G}$}}
}

\end{picture}
  \end{center}
  \caption{Interconnection of the plant $\mathbf{G}$ and the controller $\mathbf{K}$.}\label{Fig:LTI}
\end{figure}
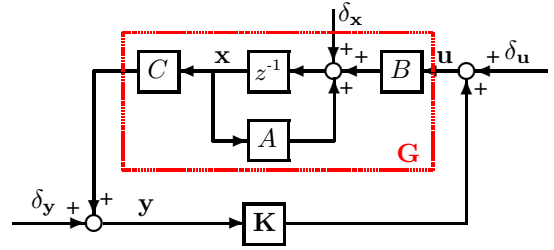
We have a standard state-space characterization for $\mathcal{C}_{\text{stab}}$.
\begin{lemma}[\!{\cite[Lemma 5.2]{zhou1996robust}}] \label{lemma:internalstability}
    $\mathbf{K}$ \emph{internally stabilizes}  $\mathbf{G}$ if and only if the following closed-loop matrix $A_{\mathrm{cl}}$ is stable.
    \vspace{-1mm}
    \begin{equation} \label{eq:closed_loop_matrix}
    A_{\mathrm{cl}} := \begin{bmatrix}
    A+BD_\mK C & BC_\mK\\
  B_\mK C & A_\mK
    \end{bmatrix}.
\end{equation}
\end{lemma}

\vspace{1mm}
The condition in \Cref{lemma:internalstability} is non-convex in $A_\mK$, $B_\mK$, $C_\mK$, $D_\mK$. It is known that if $q = n$, we can derive a convex linear matrix inequality (LMI) to characterize  $A_\mK$, $B_\mK$, $C_\mK$, $D_\mK$ by a change of variables based on Lyapunov theory~\cite{gahinet1994linear,scherer1997multiobjective,zheng2021analysis}. %; also see recent analysis on the geometry of the parameter space $A_\mK,B_\mK,C_\mK,D_\mK$~\cite{}. 

\subsection{Doubly-coprime factorization and Youla parameterization}

In addition to the state-space condition~\cref{eq:closed_loop_matrix}, there are frequency-domain characterizations for $\mathcal{C}_{\text{stab}}$, which only impose convex constraints on certain transfer functions. 
A classical approach is the celebrated \emph{Youla parameterization}~\cite{youla1976modern}, and two recent approaches are SLP~\cite{wang2019system} and IOP~\cite{furieri2019input}. As expected, Youla parameterization,  SLP,~and IOP are~equivalent~\cite{zheng2019equivalence}; see more discussions in~\cite{zheng2019systemlevel,tseng2021realization}. 
%

%Youla parameterization relies on a doubly-coprime factorization  of the plant.
 \begin{definition} \label{definition:coprime}
        A collection of stable transfer matrices, $\mathbf{U}_l$, $\mathbf{V}_l$, $\mathbf{N}_l,\mathbf{M}_l,\mathbf{U}_r, \mathbf{V}_r,\mathbf{N}_r,\mathbf{M}_r \in \mathcal{RH_{\infty}}$ is called a doubly-coprime factorization of $\mathbf{G}$ if
        $
            \mathbf{G} = \mathbf{N}_r\mathbf{M}_r^{-1} = \mathbf{M}_l^{-1}\mathbf{N}_l
        $
        and
        \vspace{-2mm}
        \begin{equation} \label{eq:coprime}
            \begin{bmatrix} \mathbf{U}_l & -\mathbf{V}_l \\ -\mathbf{N}_l & \mathbf{M}_l\end{bmatrix}
            \begin{bmatrix} \mathbf{M}_r & \mathbf{V}_r \\ \mathbf{N}_r & \mathbf{U}_r\end{bmatrix} = I.
        \end{equation}
    \end{definition}
\vspace{1mm}

    Such a doubly-coprime factorization can always be computed efficiently under \Cref{assumption:stabilizability} (see Appendix~A)~\cite{nett1984connection}. The Youla parameterization presents the equivalence~\cite{youla1976modern}
    \begin{equation} \label{eq:youla}
       % \begin{aligned}
           {\small \mathcal{C}_{\text{stab}} \!=\! \left\{\mathbf{K} \!= \!(\mathbf{V}_r \!-\! \mathbf{M}_r\mathbf{Q})(\mathbf{U}_r \!-\! \mathbf{N}_r\mathbf{Q})^{-1} \!\mid  \! \mathbf{Q} \!\in \!\mathcal{RH}_{\infty} \right\}}%\footnote{Equivalently, $\mathcal{C}_{\text{stab}}=\{(\mathbf{U}_l-\mathbf{QN}_l)^{-1}(\mathbf{V}_l-\mathbf{QM}_l)|~\mathbf{Q} \in \mathcal{RH}_\infty\}$.},
    %    \end{aligned}
    \end{equation}
    where $\mathbf{Q}$ is called the \emph{Youla parameter}. 
    The $\mathcal{RH}_{\infty}$ constraint on the Youla parameter $\mathbf{Q}$ is convex, but the order of the controller $\mathbf{K}$ cannot be specified a priori in the present form~\cref{eq:youla}.  
 The SLP~\cite{wang2019system} and the IOP~\cite{furieri2019input} require no doubly-coprime factorization, but impose a set of convex affine constraints on certain closed-loop responses. %In particular, the SLP utilizes the closed-loop responses from $(\bm{\delta}_x,\bm{\delta}_y)$ to $(\mathbf{x},\mathbf{u})$, and the IOP relies on the closed-loop responses from $(\bm{\delta}_y,\bm{\delta}_u)$ to $(\mathbf{y},\mathbf{u})$. 
 
Thanks to the convexity in the Youla, SLP, and IOP, they have found applications in distributed and robust control~\cite{zhou1996robust,rotkowitz2005characterization,furieri2019sparsity}, and recently in sample complexity analysis of learning problems~\cite{zheng2021sample,dean2020sample,zhang2021sample}.  However, the constraints on Youla, SLP, and IOP are infinitely dimensional in frequency domain, and they do not admit immediately efficient computation.  
The Ritz approximation was discussed in~\cite[Chapter 15]{boyd1991linear}, and the FIR approximation was used extensively in~\cite{wang2019system,furieri2019input,zheng2021sample,dean2020sample}. However, the Ritz or FIR approximation not only leads to large-scale~optimization problems, but also results in controllers of high-order (often much larger than the state dimension $n$); see~\cite[Section 5]{zheng2019systemlevel} for more discussions. 

\subsection{Problem statement}

The computational issue for frequency-domain characterizations of $\mathcal{C}_{\mathrm{stab}}$ has been addressed unsatisfactorily in the classical literature~\cite[Chapter 15]{boyd1991linear} or the recent studies~\cite{wang2019system,furieri2019input,zheng2021sample,dean2020sample}. This motivates the main question in this paper.

%\vspace{-1mm}
\begin{center}
    \textit{Can we develop an efficient linear matrix inequality (LMI) for a frequency-domain characterization of $\mathcal{C}_{\mathrm{stab}}$?}
\end{center}
%\vspace{-1mm}

%The main objective of this paper is to introduce an efficient LMI-based computation for a frequency-domain characterization of $\mathcal{C}_{\mathrm{stab}}$. 
We provide a positive answaer to this question. In particular, we first introduce a ``kernel'' version of the Youla parameterization~\cref{eq:youla}, which only involves one single affine constraint. This leads to a new robust filtering problem, allowing us to derive an LMI for efficient computation.

\section{Parameterization with a Single Affine Constraint and Robust filtering} \label{section:parameterization}

%In this section, we introduce a convex parameterization of  $\mathcal{C}_{\mathrm{stab}}$ with a single affine constraint and present its connection with a robust filtering problem. 

\subsection{Stabilization lemma}
We first introduce a classical stabilization lemma.
\begin{lemma} \label{lemma:stabilization}
    Given a doubly coprime factorization~\cref{eq:coprime} with $\mathbf{G} \!=\! \mathbf{N}_r \mathbf{M}^{-1}_r \!=\! \mathbf{M}_l^{-1} \mathbf{N}_l$, we have equivalent statements as
    \begin{enumerate}
        \item The controller $\mathbf{K}$ internally stabilizes $\mathbf{G}$;
        \item $\begin{bmatrix} I & -\mathbf{G} \\ - \mathbf{K} & I \end{bmatrix}^{-1} \in \mathcal{RH}_\infty$;
        \item $\begin{bmatrix} \mathbf{M}_l & -\mathbf{N}_l \\ - \mathbf{K} & I \end{bmatrix}^{-1} \in \mathcal{RH}_\infty$;
        \item $\begin{bmatrix} I & -\mathbf{N}_r \\ - \mathbf{K} & \mathbf{M}_r \end{bmatrix}^{-1} \in \mathcal{RH}_\infty$.
    \end{enumerate}
\end{lemma}
\vspace{2mm}
This result is standard~\cite[Chapter 4]{francis1987course}. 
A quick understanding might be: a classical result~\cite[Lemma 5.3]{zhou1996robust} says that $\mathbf{K}$ internally stabilizes $\mathbf{G}$ if and only if the closed-loop responses from $(\mathbf{\delta}_y,\mathbf{\delta}_u)$ to $(\mathbf{y},\mathbf{u})$ in \Cref{Fig:LTI} are stable. Simple algebra leads to 
\begin{equation} \label{eq:yutoyu}
    \begin{bmatrix} \mathbf{y} \\\mathbf{u} \end{bmatrix} = \begin{bmatrix} I & -\mathbf{G} \\ - \mathbf{K} & I \end{bmatrix}^{-1}  \begin{bmatrix} \mathbf{\delta_y} \\ \mathbf{\delta_u} \end{bmatrix}.
\end{equation}
This proves the equivalence between (1) and (2). Since
$$
\begin{aligned}
    \begin{bmatrix} I & -\mathbf{G} \\ - \mathbf{K} & I \end{bmatrix}^{-1} &=  \begin{bmatrix} \mathbf{M}_l & -\mathbf{N}_l  \\ - \mathbf{K} & I \end{bmatrix}^{-1} \begin{bmatrix} \mathbf{M}_l & 0  \\ 0  & I \end{bmatrix}  \\
    & = \begin{bmatrix} I & 0  \\ 0  & \mathbf{M}_r \end{bmatrix} \begin{bmatrix} I & -\mathbf{N}_r \\ - \mathbf{K} & \mathbf{M}_r \end{bmatrix}^{-1},
\end{aligned}
$$
this proves (3) $\Rightarrow$ (2), and (4) $\Rightarrow$ (2). The other directions are not difficult using properties of the coprime factorization~\cref{eq:coprime}.

\subsection{Convex parameterization of stabilizing controllers}
Our first result is the following convex parameterization of all stabilizing controllers, which can be considered as a ``kernel'' version of the Youla parameterization. 

\begin{theorem} \label{theorem:parameterization}
    Given a coprime factorization~\cref{eq:coprime} with  $\mathbf{G} = \mathbf{M}_l^{-1} \mathbf{N}_l$, we have an equivalent representation of $\mathcal{C}_{\text{stab}}$ as
    \begin{align} %\label{eq:single-affine-constraint}
       {\small \mathcal{C}_{\mathrm{stab}} \!=\! \{\mathbf{K} \!= \!\mathbf{Y}\mathbf{X}^{-1} \!\mid \! \mathbf{M}_l \mathbf{X} \!-\!\mathbf{N}_l \mathbf{Y}\! = \!I, \mathbf{X},\! \mathbf{Y} \!\in \!\mathcal{RH}_\infty \}. \label{eq:para_left}%\\
       % \mathcal{C}_{\text{stab}} &= \{\mathbf{K} = \mathbf{X}^{-1}\mathbf{Y} \mid  \mathbf{X}\mathbf{M}_r -\mathbf{Y} \mathbf{N}_r  = I,  \mathbf{X}, \mathbf{Y} \in \mathcal{RH}_\infty \}.   \label{eq:para_right}
       }
    \end{align}
\end{theorem}
\vspace{2mm}
\begin{proof} $\Leftarrow$ Suppose that there exist $\mathbf{X}, \mathbf{Y} \in \mathcal{RH}_\infty$ satisfying the affine constraint in~\cref{eq:para_left}. We prove that $\mathbf{K} = \mathbf{Y}\mathbf{X}^{-1}$ internally stabilizes $\mathbf{G}$. Indeed, we can verify 
    $$
    \begin{aligned}
        &\begin{bmatrix} \mathbf{M}_l & -\mathbf{N}_l \\ - \mathbf{K} & I \end{bmatrix}\begin{bmatrix} \mathbf{X} & \mathbf{X}\mathbf{N}_l \\ \mathbf{Y} & I + \mathbf{Y}\mathbf{N}_l \end{bmatrix} \\
        =&   \begin{bmatrix} \mathbf{M}_l\mathbf{X} -\mathbf{N}_l\mathbf{Y}   & (\mathbf{M}_l\mathbf{X} -\mathbf{N}_l\mathbf{Y}-I)\mathbf{N}_l \\ - \mathbf{K}\mathbf{X} + \mathbf{Y} & I + (- \mathbf{K}\mathbf{X} + \mathbf{Y})\mathbf{N}_l \end{bmatrix} 
        =  \begin{bmatrix} I & 0 \\ 0 & I \end{bmatrix}\!,
    \end{aligned}
    $$
    which means 
    $$
    \begin{bmatrix} \mathbf{M}_l & -\mathbf{N}_l \\ - \mathbf{K} & I \end{bmatrix}^{-1} = \begin{bmatrix} \mathbf{X} & \mathbf{X}\mathbf{N}_l \\ \mathbf{Y} & I + \mathbf{Y}\mathbf{N}_l \end{bmatrix} \in \mathcal{RH}_\infty.
    $$
    By \Cref{lemma:stabilization}, we know $\mathbf{K} = \mathbf{Y}\mathbf{X}^{-1} \in \mathcal{C}_{\mathrm{stab}}$.
    
    $\Rightarrow$ Given $\mathbf{K}\in \mathcal{C}_{\mathrm{stab}}$, we prove that there exist  $\mathbf{X}, \mathbf{Y} \in \mathcal{RH}_\infty$ satisfying the affine constraint in~\cref{eq:para_left} such that $\mathbf{K} = \mathbf{Y}\mathbf{X}^{-1}$. 
     By \Cref{lemma:stabilization}, we know 
     $$
        \begin{bmatrix} \mathbf{M}_l & -\mathbf{N}_l \\ - \mathbf{K} & I \end{bmatrix}^{-1} \in \mathcal{RH}_\infty.
     $$
     Upon defining 
     $$
         \begin{bmatrix} \mathbf{X} & \star \\ \mathbf{Y} & \star \end{bmatrix} := \begin{bmatrix} \mathbf{M}_l & -\mathbf{N}_l \\ - \mathbf{K} & I \end{bmatrix}^{-1},
     $$
     where $\star$ are not important, we have $\mathbf{X}, \mathbf{Y} \in \mathcal{RH}_\infty$ and 
     $$
        \mathbf{M}_l \mathbf{X} -\mathbf{N}_l \mathbf{Y} = I, \quad -\mathbf{K}\mathbf{X} + \mathbf{Y} = 0.
     $$
     This completes the proof. 
     \end{proof}

%\begin{remark}
Similarly, we can derive an equivalent parameterization using the right coprime factorization $\mathbf{G}= \mathbf{N}_r \mathbf{M}^{-1}_r $: %. Giving a right-coprime factorization , 
$$\mathcal{C}_{\text{stab}} = \{\mathbf{K} = \mathbf{X}^{-1}\mathbf{Y} \mid  \mathbf{X}\mathbf{M}_r -\mathbf{Y} \mathbf{N}_r  = I,  \mathbf{X}, \mathbf{Y} \in \mathcal{RH}_\infty \}.$$ %. %\hfill $\square$
%\end{remark}
%
There exist different internal stability conditions based on the coprime factorization~\cref{eq:coprime}; see e.g.,~\cite[Lemma 5.10 \& Corollary 5.1]{zhou1996robust}. To the best of our knowledge, the explicit characterization with a single affine constraint in~\Cref{theorem:parameterization} has not been formulated before. \Cref{theorem:parameterization}, the Youla~\cite{youla1976modern}, the SLP~\cite{wang2019system} and the IOP~\cite{furieri2019input} are expected to be equivalent among each other in theory. We give some discussions below.    

\begin{remark}[Connection with Youla]
As shown in~\cref{eq:youla}, the classical Youla parameterization only has one parameter $\mathbf{Q} \in \mathcal{RH}_\infty$ with no affine constraints. Indeed, all the solutions to the affine equation
\begin{equation} \label{eq:affine-kernel}
    \begin{bmatrix}\mathbf{M}_l & -\mathbf{N}_l \end{bmatrix} \begin{bmatrix}
\mathbf{X} \\
\mathbf{Y}
\end{bmatrix}  = I, \;\; \mathbf{X}, \mathbf{Y} \in \mathcal{RH}_\infty
\end{equation}
are parameterized by
$
\begin{bmatrix}
\mathbf{X} \\
\mathbf{Y}
\end{bmatrix} = \begin{bmatrix}
\mathbf{X}_r \\
\mathbf{Y}_r 
\end{bmatrix} + \begin{bmatrix}
\mathbf{N}_r \\
\mathbf{M}_r 
\end{bmatrix}\mathbf{Q},  \mathbf{Q} \in \mathcal{RH}_\infty,
$
where $\begin{bmatrix}
\mathbf{X}_r \\
\mathbf{Y}_r 
\end{bmatrix}$ is a special solution to~\cref{eq:affine-kernel} and $\begin{bmatrix}
\mathbf{N}_r \\
\mathbf{M}_r 
\end{bmatrix}$ is the kernel space of $\begin{bmatrix}\mathbf{M}_l & -\mathbf{N}_l \end{bmatrix}$ in $ \mathcal{RH}_\infty$, which is confirmed by the coprime factorization~\cref{eq:coprime}. 
\hfill $\square$
\end{remark}

\begin{remark}[Connection with SLP/IOP]
Both~the~SLP~and IOP utlize certain closed-loop responses to parameterize $\mathcal{C}_{\mathrm{stab}}$. In particular, the IOP~\cite{furieri2019input} relies on the closed-loop responses from $(\bm{\delta}_y,\bm{\delta}_u)$ to $(\mathbf{y},\mathbf{u})$ in~\cref{eq:yutoyu}: all internally stabilizing controllers is parameterized by $\bm{\Phi}_{yy}, \bm{\Phi}_{uy},
            \bm{\Phi}_{yu}, \bm{\Phi}_{uu}$  that lies in the affine subspace defined by 
\begin{equation}\label{eq:iop}
    \begin{aligned}
    \begin{bmatrix} I & -\mathbf{G} \end{bmatrix}\begin{bmatrix}
           \bm{\Phi}_{yy} & \bm{\Phi}_{yu} \\
            \bm{\Phi}_{uy} & \bm{\Phi}_{uu}
        \end{bmatrix} &= \begin{bmatrix} I & 0 \end{bmatrix}, \\
            \begin{bmatrix}
            \bm{\Phi}_{yy} & \bm{\Phi}_{yu} \\
            \bm{\Phi}_{uy} & \bm{\Phi}_{uu}
        \end{bmatrix}\begin{bmatrix}  -\mathbf{G} \\I \end{bmatrix} &= \begin{bmatrix} 0 \\ I\end{bmatrix}, \\
       \bm{\Phi}_{yy}, \bm{\Phi}_{uy},
            \bm{\Phi}_{yu}, \bm{\Phi}_{uu} &\in \mathcal{RH}_{\infty},
        \end{aligned}
\end{equation}
and the controller is given by $\mathbf{K} = \bm{\Phi}_{uy}\bm{\Phi}_{yy}^{-1}$. There are four affine constraints in~\cref{eq:iop}. We can verify that given any $\mathbf{X}, \mathbf{Y}$ satisfying the constraint in~\cref{eq:para_left}, the following choice 
$    \bm{\Phi}_{yy} = \mathbf{X}\mathbf{M}_l$,  $\bm{\Phi}_{uy} = \mathbf{Y}\mathbf{M}_l$, $\bm{\Phi}_{yu} = \mathbf{X}\mathbf{N}_l$,
    $\bm{\Phi}_{uu}=I + \mathbf{Y}\mathbf{N}_l
$ is feasible to~\cref{eq:iop} and parameterizes the same controller.  Similar relationship with the SLP can be derived as well.  \hfill $\square$
\end{remark}

\subsection{A robustness variant and robust $\mathcal{H}_\infty$ filtering}

While \Cref{theorem:parameterization}, Youla~\cite{youla1976modern}, SLP~\cite{wang2019system} and IOP~\cite{furieri2019input} are all equivalent with each other theoretically, they have different computational features. As we will see in \Cref{section:filtering}, the fact that \Cref{theorem:parameterization} has only one affine constraint will be essential for deriving an equivalent efficient LMI condition. %
Indeed, the single affine equality in~\cref{eq:para_left} does not need to be satisfied exactly for internal stability. %We have the following robustness lemma.
\begin{lemma}[Robustness lemma] \label{lemma:robustness}
    Given a coprime factorization~\cref{eq:coprime} with  $\mathbf{G} = \mathbf{M}_l^{-1} \mathbf{N}_l$, suppose that
\begin{equation} \label{eq:uncertainty}
    \mathbf{M}_l \mathbf{X} -\mathbf{N}_l \mathbf{Y} = I + \mathbf{\Delta},  \qquad \mathbf{X}, \mathbf{Y} \in \mathcal{RH}_\infty. 
\end{equation} If $(I + \mathbf{\Delta})^{-1} \in \mathcal{RH}_\infty$, then $\mathbf{K} = \mathbf{Y}\mathbf{X}^{-1} \in \mathcal{C}_{\mathrm{stab}}$. 
\end{lemma}
\begin{proof}
Let $\mathbf{K} \!= \!\mathbf{Y}\mathbf{X}^{-1}$ with $\mathbf{X}$ and $\mathbf{Y}$ in~\cref{eq:uncertainty}. We~have
    \begingroup
    \setlength\arraycolsep{1.5pt}
\def\arraystretch{0.9}
\small
$$
\begin{aligned}
\begin{bmatrix} \mathbf{M}_l & -\mathbf{N}_l \\ - \mathbf{K} & I \end{bmatrix}^{-1} 
\!\!=\! &\begin{bmatrix} (\mathbf{M}_l \!-\!\mathbf{N}_l\mathbf{K})^{-1}  & (\mathbf{M}_l -\mathbf{N}_l\mathbf{K})^{-1}\mathbf{N}_l \\  \mathbf{K}(\mathbf{M}_l \!-\!\mathbf{N}_l\mathbf{K})^{-1} & I \!+\! \mathbf{K}(\mathbf{M}_l \!-\!\mathbf{N}_l\mathbf{K})^{-1}\mathbf{N}_l\end{bmatrix} \\
=\!& \begin{bmatrix} \mathbf{X}(I + \mathbf{\Delta})^{-1}  & \mathbf{X}(I + \mathbf{\Delta})^{-1}\mathbf{N}_l \\  \mathbf{Y}(I + \mathbf{\Delta})^{-1} & I \!+\! \mathbf{Y}(I + \mathbf{\Delta})^{-1}\mathbf{N}_l\end{bmatrix},
\end{aligned}
$$
\endgroup
which is stable if $(I + \mathbf{\Delta})^{-1} \in \mathcal{RH}_\infty$. Combining this fact with \cref{lemma:stabilization}, we complete the proof. 
\end{proof}

\begin{remark}
The condition $(I + \mathbf{\Delta})^{-1} \in \mathcal{RH}_\infty$ is only sufficient for internal stability. Consider a simple plant 
$$\mathbf{G} = \frac{1}{z+1}, \; \mathbf{M}_l = \frac{z+1}{z}, \; \mathbf{N}_l = \frac{1}{z}.$$
We let 
$
    \mathbf{X} = 1, \mathbf{Y} = 1 \in \mathcal{RH}_\infty$ that satisfy $ \mathbf{M}_l \mathbf{X} -\mathbf{N}_l \mathbf{Y} = 1.
$
Thus, $\mathbf{K} = \mathbf{Y}\mathbf{X}^{-1} = 1$ internally stabilizes $\mathbf{G}$. Consider 
$$
     \mathbf{X}_1 \!=\! \frac{z+2}{z}, \mathbf{Y}_1 \!=\!  \frac{z+2}{z} \in \mathcal{RH}_\infty, \mathbf{M}_l \mathbf{X}_1 \! -\!\mathbf{N}_l \mathbf{Y}_1 = 1 \!+\! \frac{2}{z}.
$$
Controller $\mathbf{K} = \mathbf{Y}_1\mathbf{X}_1^{-1} = 1$ internally stabilizes $\mathbf{G}$, but
$
    (I + \mathbf{\Delta})^{-1} = \displaystyle \frac{z}{z+2}
$
is unstable. Thus, $(I + \mathbf{\Delta})^{-1} \in \mathcal{RH}_\infty$ is not necessary for internal stability \hfill $\square$
%
%A sufficient condition for $(I + \mathbf{\Delta})^{-1} \in \mathcal{RH}_\infty$ is $\|\Delta\|_\infty < 1$.
\end{remark}

From~\Cref{lemma:robustness}, we are ready to introduce our second result that can be interpreted as a robust filtering problem.
\begin{theorem} \label{theorem:internal_stability_hinf}
     Given a coprime factorization~\cref{eq:coprime} with  $\mathbf{G} = \mathbf{M}_l^{-1} \mathbf{N}_l$, the controller $\mathbf{K}$ internally stabilizes $\mathbf{G}$ if and only if there exist $\mathbf{X}$ and $\mathbf{Y}$ in $\mathcal{RH}_\infty$ such that 
     \begin{equation} \label{eq:affine_residual}
         \|\mathbf{M}_l \mathbf{X} -\mathbf{N}_l \mathbf{Y} - I \|_\infty < \epsilon < 1. 
     \end{equation}
     If~\cref{eq:affine_residual} holds, then $\mathbf{K} = \mathbf{Y}\mathbf{X}^{-1}$ is an internally stabilizing controller and the closed-loop response satisfies
     \begin{equation} \label{eq:error_bound}
     \begin{aligned}
         &\left\|\begin{bmatrix} I & -\mathbf{G} \\ -\mathbf{K} & I \end{bmatrix}^{-1} - \begin{bmatrix} \mathbf{X}\mathbf{M}_l & \mathbf{X}\mathbf{N}_l \\ \mathbf{Y}\mathbf{M}_l & I + \mathbf{Y}\mathbf{N}_l \end{bmatrix} \right\|_\infty \\
         \leq \; &\frac{\epsilon}{1-\epsilon} \left\|\begin{bmatrix} \mathbf{X} \\ \mathbf{Y} \end{bmatrix}\right\|_\infty \left\|\begin{bmatrix} \mathbf{M}_l & \mathbf{N}_l \end{bmatrix}\right\|_\infty.
          \end{aligned}
     \end{equation}
     \vspace{2mm}
%     where $\epsilon \in (0,1)$ is such that $\|\mathbf{M}_l \mathbf{X} -\mathbf{N}_l \mathbf{Y} - I \|_\infty < \epsilon$. 
\end{theorem}

\begin{proof}
$\Rightarrow$ If $\mathbf{K}$ internally stabilizes $\mathbf{G}$, \Cref{theorem:parameterization} guarantees that we have $\mathbf{X}, \mathbf{Y} \in \mathcal{RH}_\infty$ such that $\mathbf{K} = \mathbf{Y}\mathbf{X}^{-1}$ and $\mathbf{M}_l \mathbf{X} -\mathbf{N}_l \mathbf{Y} = I$. Thus,~\cref{eq:affine_residual} is trivially satisfied. 

$\Leftarrow$ Let $\mathbf{X}, \mathbf{Y} \in \mathcal{RH}_\infty$ satisfy~\cref{eq:affine_residual}. Then 
$$
    \mathbf{\Delta} := \mathbf{M}_l \mathbf{X} -\mathbf{N}_l \mathbf{Y} - I \in \mathcal{RH}_\infty, \quad \|\mathbf{\Delta}\|_\infty < 1. 
$$
By the small gain theorem, we know $(1\!+\!\mathbf{\Delta})^{-1} \!\in\! \mathcal{RH}_\infty$. \Cref{lemma:robustness} implies that $\mathbf{K} \!= \!\mathbf{Y}\mathbf{X}^{-1}$ internally stabilizes $\mathbf{G}$.

To prove~\cref{eq:error_bound},  applying $\mathbf{K} = \mathbf{Y}\mathbf{X}^{-1}$ leads to the closed-loop response as 
$$
    \begin{bmatrix} I & -\mathbf{G} \\ -\mathbf{K} & I \end{bmatrix}^{-1}\!\! = \!\begin{bmatrix} \mathbf{X}(I + \mathbf{\Delta})^{-1}\mathbf{M}_l  & \mathbf{X}(I + \mathbf{\Delta})^{-1}\mathbf{N}_l \\  \mathbf{Y}(I + \mathbf{\Delta})^{-1}\mathbf{M}_l & I + \mathbf{Y}(I + \mathbf{\Delta})^{-1}\mathbf{N}_l\end{bmatrix}\!.
$$
Considering $(I + \mathbf{\Delta})^{-1} = I - \mathbf{\Delta}(I + \mathbf{\Delta})^{-1}$, 
it is easy to verify that 
$$
\begin{aligned}
&\begin{bmatrix} I & -\mathbf{G} \\ -\mathbf{K} & I \end{bmatrix}^{-1} - \begin{bmatrix} \mathbf{X}\mathbf{M}_l & \mathbf{X}\mathbf{N}_l \\ \mathbf{Y}\mathbf{M}_l & I + \mathbf{Y}\mathbf{N}_l \end{bmatrix} \\
= &- \begin{bmatrix} \mathbf{X} \\ \mathbf{Y} \end{bmatrix}\mathbf{\Delta}(I + \mathbf{\Delta})^{-1}\begin{bmatrix} \mathbf{M}_l & \mathbf{N}_l \end{bmatrix}.
\end{aligned}
$$
In addition, we have the following $\mathcal{H}_\infty$ norm inequalities 
$$
\begin{aligned}
%    \|(I + \mathbf{\Delta})^{-1}\|_\infty \leq \frac{1}{1 - \epsilon},\\
    \|\mathbf{\Delta}(I + \mathbf{\Delta})^{-1}\|_\infty \leq  \|\mathbf{\Delta}\|_\infty \|(I + \mathbf{\Delta})^{-1}\|_\infty \leq \frac{\epsilon}{1 - \epsilon},
\end{aligned}
$$
and thus~\cref{eq:error_bound} follows. 
\end{proof}

We note that the condition~\cref{eq:affine_residual} has an interesting interpretation as a \textit{robust filtering} problem~\cite{geromel2000h,geromel2002robust}: it aims to find a stable filter  $\begin{bmatrix}
\mathbf{X} & \mathbf{Y}
\end{bmatrix} \in \mathcal{RH}_\infty$ such that the residual $\mathbf{M}_l \mathbf{X} -\mathbf{N}_l \mathbf{Y} - I$ has $\mathcal{H}_\infty$ norm less than 1. This filtering interpretation motivates the LMI development in~\Cref{section:filtering}. 

%{\color{red} Some comments about the filtering; a figure about the filtering setup}

\section{LMI-based Computation via Filtering} \label{section:filtering}

\subsection{$\mathcal{H}_\infty$ filtering problem}

We consider a \textbf{right $\mathcal{H}_\infty$ filtering} problem: given  $\mu > 0$ and $\mathbf{P}_1(z), \mathbf{P}_2(z) \in \mathcal{RH}_\infty$ with a state-space realization 
$$
    \begin{bmatrix} \mathbf{P}_1(z) & \mathbf{P}_2(z) \end{bmatrix} = \left [
        \hspace{-.2ex}
        \begin{array}{c|cc}
          A & B_1 & B_2 \\ \hline
          C & D_1 & D_2
        \end{array}
        \hspace{-.2ex}
        \right ],
$$
find a stable filter $\mathbf{F}(z) \in \mathcal{RH}_\infty$ such that 
\begin{equation} \label{eq:hinf_filter}
 \| \mathbf{P}_1(z) \mathbf{F}(z) - \mathbf{P}_2(z) \|_\infty < \mu.
\end{equation}

We call~\cref{eq:hinf_filter} as the \textit{right} $\mathcal{H}_\infty$ filtering problem, since the filter $\mathbf{F}(z)$ is on the right side of $\mathbf{P}_1(z)$. In the classical literature on filtering (see~\cite{geromel2000h,geromel2002robust} and the references therein), a \textit{left} $\mathcal{H}_\infty$ filtering problem is more common: find $\mathbf{F}(z) \in \mathcal{RH}_\infty$ such that 
\begin{equation} \label{eq:hinf_filter_dual}
 \| \mathbf{F}(z)\mathbf{H}_1(z)  - \mathbf{H}_2(z) \|_\infty < \mu.
\end{equation}
\Cref{Fig:filtering} illustrates these two types of filtering problems. It seems that most existing literature focuses on the {left} $\mathcal{H}_\infty$ filtering problem~\cref{eq:hinf_filter_dual}, while the right $\mathcal{H}_\infty$ filtering problem~\cref{eq:hinf_filter} has received less attention. Therefore, our LMI-based solution to~\cref{eq:hinf_filter} might be of independent interest. 

%The solution to the right $\mathcal{H}_\infty$ filtering problem

%{\color{red} Some comments that this problem seems not standard, as the standard filtering is on the left.}

\begin{figure}[t]
\setlength{\belowcaptionskip}{0pt}
\setlength{\abovecaptionskip}{9pt}
  \centering
  \begin{adjustbox}{minipage=\linewidth,scale=0.9}
  %\subfigure[]
  {\setlength{\unitlength}{0.008in}
\begin{picture}(200,110)(140,410)
\thicklines

% lines 
\put(234.5,440){\line(-1, 0){ 55}}
\put(180,510){\vector( 0, -1){ 31.5}}
\put(179.5,475){\circle{8}}
\put(180,440){\vector( 0, 1){ 31.5}}
%\put(181,440){\circle{10}}%\put(340,550){\line( 0, 1){ 70}}
\put(179.5,510){\line( 1, 0){ 24}}
\put(175,475){\vector(-1,0){ 30}}

\put(315.5,510){\vector( -1, 0){ 30}}
\put(253.5,510){\vector( -1, 0){ 18}}
\put(315.5,440){\vector(-1, 0){ 50}}
%\put(315.5,475){\circle{8}}
\put(315.5,475){\line(1,0){ 20}}
\put(315.5,510){\line( 0,-1){70}}
%\put(315.5,439.5){\line( -1,0){ 30.5}}
%\put(315.5,510){\line( 0,-1){ 30.5}}

% Dynamical components
\put(235,424){\framebox(30,30){}}
\put(241,431){\makebox(0,0)[lb]{$\mathbf{P}_2$}}

\put(255,494){\framebox(30,30){}}
\put(263,503){\makebox(0,0)[lb]{$\mathbf{F}$}}

\put(205,494){\framebox(30,30){}}
\put(211,501){\makebox(0,0)[lb]{$\mathbf{P}_1$}}

% signals 
\put(200,445){\makebox(0,0)[lb]{$\mathbf{y}_2$}}
\put(185,515){\makebox(0,0)[lb]{$\mathbf{y}_1$}}
\put(150,480){\makebox(0,0)[lb]{$\mathbf{y}$}}
\put(325,480){\makebox(0,0)[lb]{$\mathbf{u}$}}
%\put(285,415){\makebox(0,0)[lb]{$\mathbf{u}$}}

\put(185,483){\makebox(0,0)[lb]{\small +}}
\put(188,460){\makebox(0,0)[lb]{\small -}}

%\put(185,450){\makebox(0,0)[lb]{\tiny +}}
%\put(167,445){\makebox(0,0)[lb]{\tiny +}}

%\put(350,515){\makebox(0,0)[lb]{\tiny +}}
%\put(345,498){\makebox(0,0)[lb]{\tiny +}}
\put(240,400){\makebox(0,0)[lb]{\small (a)}}
\end{picture}}
  %\end{adjustbox}
 % \hspace{-40mm}
  %\begin{adjustbox}{minipage=\linewidth,scale=0.8}
 % \subfigure[]
  {\setlength{\unitlength}{0.008in}
\begin{picture}(200,115)(140,410)
\thicklines

% lines 
\put(234.5,440){\line(-1, 0){ 55}}
\put(180,510){\vector( 0, -1){ 31.5}}
\put(179.5,475){\circle{8}}
\put(180,440){\vector( 0, 1){ 31.5}}
%\put(181,440){\circle{10}}%\put(340,550){\line( 0, 1){ 70}}
\put(179.5,510){\line( 1, 0){ 24}}
\put(175,475){\vector(-1,0){ 30}}

\put(315.5,510){\vector( -1, 0){ 30}}
\put(253.5,510){\vector( -1, 0){ 18}}
\put(315.5,440){\vector(-1, 0){ 50}}
%\put(315.5,475){\circle{8}}
\put(315.5,475){\line(1,0){ 20}}
\put(315.5,510){\line( 0,-1){70}}
%\put(315.5,439.5){\line( -1,0){ 30.5}}
%\put(315.5,510){\line( 0,-1){ 30.5}}

% Dynamical components
\put(235,424){\framebox(30,30){}}
\put(241,431){\makebox(0,0)[lb]{$\mathbf{H}_2$}}

\put(255,494){\framebox(30,30){}}
\put(261,501){\makebox(0,0)[lb]{$\mathbf{H}_1$}}

\put(205,494){\framebox(30,30){}}
\put(213,503){\makebox(0,0)[lb]{$\mathbf{F}$}}

% signals 
\put(200,445){\makebox(0,0)[lb]{$\mathbf{y}_2$}}
\put(185,515){\makebox(0,0)[lb]{$\mathbf{y}_1$}}
\put(150,480){\makebox(0,0)[lb]{$\mathbf{y}$}}
\put(325,480){\makebox(0,0)[lb]{$\mathbf{u}$}}
%\put(285,415){\makebox(0,0)[lb]{$\mathbf{u}$}}

\put(185,483){\makebox(0,0)[lb]{\small +}}
\put(188,460){\makebox(0,0)[lb]{\small -}}

\put(240,400){\makebox(0,0)[lb]{\small (b)}}
%\put(185,450){\makebox(0,0)[lb]{\tiny +}}
%\put(167,445){\makebox(0,0)[lb]{\tiny +}}

%\put(350,515){\makebox(0,0)[lb]{\tiny +}}
%\put(345,498){\makebox(0,0)[lb]{\tiny +}}

\end{picture}}
  \end{adjustbox}
  \caption{(a) Right-filtering problem, where the filter $\mathbf{F}$ appears \textit{before} the dynamical system $\mathbf{P}_1$. (b) Left-filtering problem, where the filter $\mathbf{F}$ appears \textit{after} the dynamical system $\mathbf{H}_1$}\label{Fig:filtering}
\end{figure}
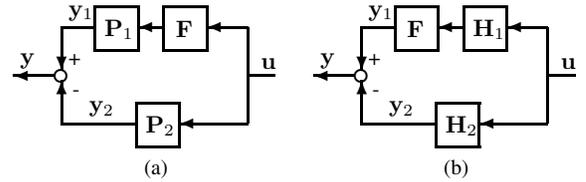

%We introduce the standard lemma for computing $\mathcal{H}_\infty$ norm of stable transfer functions. 
\begin{lemma} \label{lemma:Hinf}
    Given a stable transfer function $\mathbf{T}(z) = C(zI - A)^{-1}B + D \in \mathcal{RH}_\infty$, then $\|\mathbf{T}(z) \|^2_\infty < \mu$ if and only if there exists a positive definite matrix $P \succ 0$ such that 
    \vspace{-2mm}
    \begin{equation} \label{eq:HinfLMI}
        \begin{bmatrix} P & AP & B & 0 \\ PA^\tr & P & 0 & PC^\tr  \\ B^\tr & 0 & I & D^\tr \\ 0 & CP & D & \mu I \end{bmatrix} \succ 0.
    \end{equation}
\end{lemma}

\vspace{1mm}

The right $\mathcal{H}_\infty$ filtering is solved in the theorem below. 
\begin{theorem} \label{theorem:hinf_filtering}
    There exists $\mathbf{F}(z) \in \mathcal{RH}_\infty$ such that~\cref{eq:hinf_filter} holds if and only if there exist symmetric matrices $X, Z$, and matrices $Q, F, L, R$ of compatible dimensions such that 
    \begingroup
    \setlength\arraycolsep{0.8pt}
\def\arraystretch{0.9}
   \begin{align} \label{eq:filter_LMI}
      {\small \begin{bmatrix}
        X & Z & A X\!\! +\!\! B_1 L & A Z \!\!+\!\! B_1 L & B_1 R \!-\! \!B_2 & 0 \\
        \star & Z & Q & Q & F & 0 \\
        \star &\star  & X & Z & 0 & X C^\tr\!\! +\! L^\tr D_1^\tr \\
        \star & \star & \star & Z & 0 & Z C^\tr\! \!+\! L^\tr D_1^\tr \\
        \star & \star & \star & \star & I & R^\tr D_1^\tr \!-\! D_2^\tr \\
        \star & \star & \star & \star & \star & \mu^2 I
      \end{bmatrix} \!\succ\! 0,\!}
    \end{align}
    \endgroup
    where $\star$ denotes the symmetric parts. If~\cref{eq:filter_LMI} holds, a state-space realization of $\mathbf{F}(z) = \hat{C}(zI - \hat{A})^{-1}\hat{B} + \hat{D}$ is
         \begin{align} \label{eq:filter_LMI_state_space}
     \begin{bmatrix}
      \hat{A} & \hat{B} \\
      \hat{C} & \hat{D}
    \end{bmatrix} = \begin{bmatrix}
      U^\tr Z^{-1} & 0 \\
      0 & I
    \end{bmatrix}\begin{bmatrix}
      Q & F \\
      L & R
    \end{bmatrix}\begin{bmatrix}
      U^{-1} & 0 \\
      0 & I
    \end{bmatrix}^\tr,% \\
    %Z &:= - N V U^\tr = X - N.
\end{align}
    where $U$ is an arbitrary non-singular matrix. 
\end{theorem}

\begin{proof}
    Let a state-space realization of $\mathbf{F}(z)$ be 
    $
        \mathbf{F}(z) = \left [
        \hspace{-.2ex}
        \begin{array}{c|c}
          {\hat{A}} & {\hat{B}} \\ \hline
          {\hat{C}} & {\hat{D}}
        \end{array}
        \hspace{-.2ex}
        \right ].
    $
    Standard system operations (see Appendix C) lead to the following  state-space realization 
    $$
    \begin{aligned}
        \mathbf{P}_1(z) \mathbf{F}(z) - \mathbf{P}_2(z)&= \left [
      \hspace{-.2ex}
      \begin{array}{cc|c}
        A & B_1 {\hat{C}} & B_1 {\hat{D}} - B_2 \\
        0 & {\hat{A}} & {\hat{B}} \\ \hline
        C & D_1 {\hat{C}} & D_1 {\hat{D}} - D_2
      \end{array}
      \hspace{-.2ex}
      \right ] \\
      &= : \left [
      %\hspace{-.2ex}
      \begin{array}{c|c}
        \tilde{A} & \tilde{B} \\ \hline
        \tilde{C} & \tilde{D}
      \end{array}
      \hspace{-.2ex}
      \right ] . 
    \end{aligned}
    $$
    By \Cref{lemma:Hinf}, we know~\cref{eq:hinf_filter} holds if and only if there exists a positive definite matrix $\tilde{P}$ such that 
    \begin{equation} \label{eq:HinfLMI_s1}
        \begin{bmatrix} \tilde{P} & \tilde{A}\tilde{P} & \tilde{B} & 0 \\ \tilde{P}\tilde{A}^\tr & \tilde{P} & 0 & \tilde{P}\tilde{C}^\tr  \\ \tilde{B}^\tr & 0 & I & \tilde{D}^\tr \\ 0 & \tilde{C}\tilde{P} & \tilde{D} & \mu^2 I \end{bmatrix} \succ 0.
    \end{equation}
    Note that \cref{eq:HinfLMI_s1} is bilinear in terms of the design variable $\tilde{P}$ and the filter realization $\hat{A}, \hat{B}, \hat{C}, \hat{D}$. Motivated by the nonlinear change of variables in~\cite{geromel2000h,scherer1997multiobjective}, we partition the Lyapunov variable $\tilde{P}$ 
%       \begin{align*}
%     \tilde{P} &=
%     \begin{bmatrix}
%       P_{11} & P_{12} \\
%       P_{12}^\tr & P_{22}
%     \end{bmatrix}, &
%     \tilde{X} &=
%     \begin{bmatrix}
%       X & \hat{U} \\
%       U & \hat{X}
%     \end{bmatrix}, &
%     \tilde{X}^{-1} &=
%     \begin{bmatrix}
%       N^{-1} & \hat{V}\tr \\
%       V\tr & \hat{Y}\tr
%     \end{bmatrix}
%   \end{align*}
and its inverse as 
   \begingroup
    \setlength\arraycolsep{1.3pt}
\def\arraystretch{1}
$$
    \tilde{P} := \begin{bmatrix} X & U \\ U^\tr & \hat{X} \end{bmatrix}, %\in \mathbb{S}^{n +n_f}_+,
    \quad \tilde{P}^{-1} := \begin{bmatrix} Y & V \\ V^\tr & \hat{Y} \end{bmatrix}. %\in \mathbb{S}^{n +n_f}_+.
$$
\endgroup
Since $\tilde{P}\tilde{P}^{-1} = I$, we have 
\begin{equation} \label{eq:inverse_equality}
    XY + UV^\tr = I.
\end{equation}

Let $\hat{A}$ and $A$ have the same dimension, then $U$ and $V$ are invertible. We define $N: = Y^{-1}$ and %a transformation matrix 
$$
    \tilde{T} := \begin{bmatrix}
    I & 0 \\ 0 & -V^\tr N
    \end{bmatrix}.
$$
We further define a change of variables $Z := - N V U^\tr = X - N$ (derived from~\cref{eq:inverse_equality}), which is symmetric, and 
%\begin{subequations}
\begin{align} \label{eq:change_of_variables}
     \begin{bmatrix}
      Q & F \\
      L & R
    \end{bmatrix}
     &:=
    \begin{bmatrix}
      - N V & 0 \\
      0 & I
    \end{bmatrix}
    \begin{bmatrix}
      \hat{A} & \hat{B} \\
      \hat{C} & \hat{D}
    \end{bmatrix}
    \begin{bmatrix}
      U^\tr & 0 \\
      0 & I
    \end{bmatrix}. % \\
    %Z &:= - N V U^\tr = X - N.
\end{align}
%\end{subequations}
%
We can then verify that (some detailed computations are presented in the appendix)
\begin{subequations} \label{eq:transformation}
\begin{align}
 \tilde{T}^\tr \tilde{P} \tilde{T} &=
    \begin{bmatrix}
      X & Z \\
      Z & Z
    \end{bmatrix}, \label{eq:transformation_a}\\ 
    \tilde{T}^\tr \tilde{A} \tilde{P} \tilde{T}
    &= \begin{bmatrix}
      A X + B_1 L & AZ + B_1 L \\
      Q & Q
    \end{bmatrix},  \label{eq:transformation_b}\\
    \tilde{T}^\tr \tilde{B} &=
    \begin{bmatrix}
      B_1 R - B_2 \\
      F
    \end{bmatrix}, \label{eq:transformation_c} \\
    \tilde{C} \tilde{P} \tilde{T}
    &= \begin{bmatrix}
      C X + D_1 L & C Z + D_1 L
    \end{bmatrix}.  \label{eq:transformation_d}
\end{align}
\end{subequations}
Then,~\cref{eq:HinfLMI_s1} is equivalent to
\begingroup
    \setlength\arraycolsep{2pt}
\def\arraystretch{0.9}
$$
 {\begin{bmatrix} \tilde{T} & & & \\ & \tilde{T} & & \\ & & I & \\ & &  & I \end{bmatrix}^\tr \begin{bmatrix} \tilde{P} & \tilde{A}\tilde{P} & \tilde{B} & 0 \\ \tilde{P}\tilde{A}^\tr & \tilde{P} & 0 & \tilde{P}\tilde{C}^\tr  \\ \tilde{B}^\tr & 0 & I & \tilde{D}^\tr \\ 0 & \tilde{C}\tilde{P} & \tilde{D} & \mu^2 I \end{bmatrix}\begin{bmatrix} \tilde{T} & & & \\ & \tilde{T} & & \\ & & I & \\ & &  & I \end{bmatrix} \succ 0,}
$$
\endgroup
which turns out to be the same as~\cref{eq:filter_LMI}. From~\cref{eq:change_of_variables}, the state-space realization of $\mathbf{F}(z)$ is
$$
    \begin{bmatrix}
      \hat{A} & \hat{B} \\
      \hat{C} & \hat{D}
    \end{bmatrix} = \begin{bmatrix}
      - N V & 0 \\
      0 & I
    \end{bmatrix}^{-1}\begin{bmatrix}
      Q & F \\
      L & R
    \end{bmatrix} \begin{bmatrix}
      U^\tr & 0 \\
      0 & I
    \end{bmatrix}^{-1}.
$$
We only need to prove $-(NV)^{-1} = U^\tr Z^{-1}$, which is the same as (note that the last equation is~\cref{eq:inverse_equality})
$$
\begin{aligned}
    V^{-1}Y + U^\tr Z^{-1} = 0 &\quad \Leftrightarrow  \quad  Y(X - Y^{-1})  + VU^\tr = 0  \\
   & \quad \Leftrightarrow  \quad  YX - I  + VU^\tr = 0, \\
   & \quad \Leftrightarrow  \quad  XY + UV^\tr = I,
\end{aligned}
$$
where the first equivalence applied the fact that $Z = X - Y^{-1}$. 
This completes the proof. 
\end{proof}

The linearization of the bilinear inequality~\cref{eq:HinfLMI_s1} via the nonlinear change of variables in~\cref{eq:change_of_variables} and~\cref{eq:transformation} is motivated by the classical literature on robust filtering~\cite{geromel2000h,geromel2002robust}. 
% Similar linearization idea can also be traced back to the seminal work on LMI computation of $\mathcal{H}_\infty$ control~\cite{gahinet1994linear}. 
%
Due to the difference between \textit{right} and \textit{left} filtering problems, we remark that the LMI characterization in~\cref{eq:filter_LMI} has not appeared in~\cite{geromel2000h,geromel2002robust}, and thus \Cref{theorem:hinf_filtering} might have independent interest.  Note that we have used the standard $\mathcal{H}_\infty$ LMI in \Cref{eq:HinfLMI}, and that one can further derive a similar LMI to solve~\cref{eq:hinf_filter} based on the extended $\mathcal{H}_\infty$ LMI in~\cite{de2002extended}. We provide such a characterization in Appendix D. 

%{\color{red} Some comments on filtering problem; comparison with the literature}

\subsection{Enforcing internal stability via an LMI}

From \Cref{theorem:hinf_filtering}, we can derive an equivalent LMI formulation for the internal stability condition in~\Cref{theorem:internal_stability_hinf}. This is formally stated in the theorem below.

\begin{theorem}\label{theorem:feas_lmi}
      Given a coprime factorization~\cref{eq:coprime} with  $\mathbf{G} = \mathbf{M}_l^{-1} \mathbf{N}_l$. Let $\mathbf{M}_l$ and $\mathbf{N}_l$ have the state-space realization
     \begin{equation} \label{eq:coprime-realization}
         \begin{bmatrix}
        \mathbf{M}_l(z) & \mathbf{N}_l(z)
      \end{bmatrix}
      =
    \left [
      \hspace{-.2ex}
      \begin{array}{c|cc}
        A & B_M & B_N \\ \hline
        C & D_M & D_N
      \end{array}
      \hspace{-.2ex}
      \right ].
     \end{equation}
       There exist $\mathbf{X}(z)$ and $\mathbf{Y}(z)$ in $\mathcal{RH}_\infty$ such that
    \begin{equation}
      \| \mathbf{M}_l(z) \mathbf{X}(z) - \mathbf{N}(z) \mathbf{Y}(z) - I\|_\infty < \epsilon
    \end{equation}
    
    \begin{figure*}[b]
    \begin{equation}
    \begin{aligned} \label{eq:stabilization_LMI}
     {\small \begin{bmatrix}
        X & Z & f_1(X,L_X, L_Y) & f_2(Z,L_X, L_Y)
         & f_3(R_X,R_Y) & 0 \\
         \star & Z & Q & Q & F & 0 \\
         \star&  \star & X & Z & 0 & f_4(X, L_X, L_Y) \\
         \star &  \star&  \star & Z & 0 & f_5(Z, L_X, L_Y)  \\
         \star &  \star&  \star&  \star& I & f_6(R_X,R_Y) \\
         \star&  \star& \star & \star & \star & {\epsilon}^2 I
      \end{bmatrix} \succ 0.}
    \end{aligned}
    \end{equation}
\end{figure*}
    
    \noindent if and only if there exist symmetric matrices $X$, $Z$, and matrices $Q$, $F$, $L_X$, $L_Y$, $R_X$ and
    $R_Y$ of compatible size such that the LMI~\cref{eq:stabilization_LMI} holds. In~\cref{eq:stabilization_LMI}, notation $\star$ denotes the symmetric parts and $f_i, i = 1, \ldots, 6$ are linear functions as
    $$
    \begin{aligned}
        f_1(X,L_X, L_Y) &= A X + B_M L_X - B_N L_Y, \\
        f_2(Z,L_X,L_Y) &= A Z + B_M L_X - B_N L_Y, \\
        f_3(R_X, R_Y)&=B_M R_X - B_N R_Y, \\
        f_4(X, L_X, L_Y) &= X^\tr C^\tr + L_X^\tr D_M^\tr
        - L_Y^\tr D_N^\tr, \\
        f_5(Z, L_X, L_Y) &= Z^\tr C^\tr + L_X^\tr D_M^\tr
        - L_Y^\tr D_N^\tr,\\
        f_6(R_X, R_Y)  &= R_X^\tr D_M^\tr - R_Y^\tr D_N^\tr - I.\\
       % f_4(X, L_X, L_Y) &= X^\tr C^\tr + L_X^\tr D_M^\tr
       % - L_Y^\tr D_N^\tr\\ 
        %f_5(Z, L_X, L_Y) &= Z^\tr C^\tr + L_X^\tr D_M^\tr
        %- L_Y^\tr D_N^\tr\\
        %f_6(R_X, R_Y)  &= R_X^\tr D_M^\tr - R_Y^\tr D_N^\tr - I
    \end{aligned}
    $$
%     \begingroup
%     \setlength\arraycolsep{1pt}
% \def\arraystretch{0.9}
%     \begin{align} \label{eq:stabilization_LMI}
%      {\small \begin{bmatrix}
%         X & Z & f_1(X,L_X, L_Y) & f_2(Z,L_X, L_Y)
%          & f_3(R_X,R_Y) & 0 \\
%          \star & Z & Q & Q & F & 0 \\
%          \star&  \star & X & Z & 0 & f_4(X, L_X, L_Y) \\
%          \star &  \star&  \star & Z & 0 & f_5(Z, L_X, L_Y)  \\
%          \star &  \star&  \star&  \star& I & f_6(R_X,R_Y) \\
%          \star&  \star& \star & \star & \star & \epsilon I
%       \end{bmatrix} \succ 0,}
%     \end{align}
%     \endgroup
    %
    If~\cref{eq:stabilization_LMI} holds, state-space realizations for $\mathbf{X}(z)$ and $\mathbf{Y}(z)$ are
    % $$ 
    % \begin{aligned}
    % \mathbf{X}(z) &= \hat{C}_X(zI -  \hat{A})^{-1}\hat{B} + \hat{D}_X, \\
    % \mathbf{Y}(z) &= \hat{C}_Y(zI -  \hat{A})^{-1}\hat{B} + \hat{D}_Y, 
    % \end{aligned}
    % $$ 
    \begin{align} \label{eq:X_Y_realization_stabilization}
      \begin{bmatrix}
        \mathbf{X}(z) \\ \mathbf{Y}(z)
      \end{bmatrix}
      =
      \left [
        \hspace{-.2ex}
        \begin{array}{c|c}
          \hat{A} & \hat{B} \\ \hline
          \hat{C}_X & \hat{D}_X \\
          \hat{C}_Y & \hat{D}_Y
        \end{array}
        \hspace{-.2ex}
        \right ]
      \end{align}
      where 
      \begin{equation*} 
        \begin{bmatrix}
          \hat{A} & \hat{B} \\ 
          \hat{C}_X & \hat{D}_X \\
          \hat{C}_Y & \hat{D}_Y
        \end{bmatrix} = 
        \begin{bmatrix}
        U^\tr Z^{-1} & 0 & 0 \\
        0 & I & 0 \\
        0 & 0 & I
        \end{bmatrix}
        \begin{bmatrix}
        Q & F \\ 
        L_X & R_X \\
        L_Y & R_Y 
        \end{bmatrix}
        \begin{bmatrix}
        U^{-1} & 0 \\
        0 & I
        \end{bmatrix}^\tr
      \end{equation*}
 and $U$ is an arbitrary
    non-singular matrix.
\end{theorem}

\begin{proof}
Define $\mathbf{P}_1(z) = \begin{bmatrix}
        \mathbf{M}_l(z) & -\mathbf{N}_l(z)
      \end{bmatrix}$, and $\mathbf{P}_2(z) = I$, which have a state-space realization as 
      $$
        \begin{bmatrix}
        \mathbf{P}_1(z) & \mathbf{P}_2(z)
        \end{bmatrix} = \left [
        \hspace{-.2ex}
        \begin{array}{c|cc}
          A & B_1 & 0 \\ \hline
          C & D_1 & I
        \end{array}
        \hspace{-.2ex}
        \right ],
      $$
      with $B_1 = \begin{bmatrix} B_M & - B_N \end{bmatrix}$ and $D_1 = \begin{bmatrix} D_M & - D_N \end{bmatrix}$. 
      We let 
       \begin{align*}
      \mathbf{F}(z) &=
      \begin{bmatrix}
        \mathbf{X}(z) \\ \mathbf{Y}(z)
      \end{bmatrix}
      =
      \left [
        \hspace{-.2ex}
        \begin{array}{c|c}
          \hat{A} & \hat{B} \\ \hline
          \hat{C}_X & \hat{D}_X \\
          \hat{C}_Y & \hat{D}_Y
        \end{array}
        \hspace{-.2ex}
        \right ].
    \end{align*}
    Applying~\Cref{theorem:hinf_filtering} to $\|\mathbf{P}_1(z)\mathbf{F}(z) - \mathbf{P}_2(z)\|_\infty < \epsilon$ completes the proof. 
\end{proof}

Setting $\epsilon = 1$ recovers the internal stability condition~\cref{eq:affine_residual} in \Cref{theorem:internal_stability_hinf}. Thus, the following corollary is immediate. 
\begin{corollary} \label{corollary:stability}
Given a coprime factorization~\cref{eq:coprime} with  $\mathbf{G} = \mathbf{M}_l^{-1} \mathbf{N}_l$ and \cref{eq:coprime-realization}, the controller $\mathbf{K}$ internally stabilizes $\mathbf{G}$ if and only if there exist symmetric matrices $X$, $Z$, and matrices $Q$, $F$, $L_X$, $L_Y$, $R_X$ and
    $R_Y$ of compatible size such that the LMI~\cref{eq:stabilization_LMI} holds with $\epsilon = 1$. If~\cref{eq:stabilization_LMI} holds with $\epsilon = 1$, the following controller $ \mathbf{K}$ internally stabilizes $\mathbf{G}$,
    \begin{equation} \label{eq:state-space-K}
        \mathbf{K} = \mathbf{Y}\mathbf{X}^{-1} = \left [
      %\hspace{-.2ex}
      \begin{array}{c|c}
        \hat{A} - \hat{B}\hat{D}_X^{-1}\hat{C}_X & -\hat{B}\hat{D}_X^{-1} \\  \hline 
       -\hat{C}_Y+ \hat{D}_Y \hat{D}_X^{-1}\hat{C}_X & \hat{D}_Y\hat{D}_X^{-1}
      \end{array}
      \hspace{-.2ex}
      \right ],
    \end{equation}
     where $\mathbf{Y}$ and $\mathbf{X}$ have state-space realizations in~\cref{eq:X_Y_realization_stabilization}.  
\end{corollary}

The state-space realization of $ \mathbf{K} = \mathbf{Y}\mathbf{X}^{-1} $~\cref{eq:state-space-K} is based on standard system operations (see, e.g.,~\cite[Chapter 3.6]{zhou1996robust}). We provide a detailed calculation for~\cref{eq:state-space-K} in Appendix C. Note that the state-space realization~\cref{eq:coprime-realization} for $\mathbf{M}_l$ and $\mathbf{N}_l$ can be easily computed under \Cref{assumption:stabilizability} (see Appendix A).  

\begin{remark}[Comparison with Youla/SLP/IOP]  Youla~\cite{youla1976modern}, SLP~\cite{wang2019system}, IOP~\cite{furieri2019input} and \Cref{theorem:parameterization} present equivalent convex parameterizations for $\mathcal{C}_{\mathrm{stab}}$. However, they have very different numerical features in practical computation. The Youla parameter $\mathbf{Q}$ can be freely chosen in $\mathcal{RH}_\infty$, but the resulting controller in~\cref{eq:youla} may not have \textit{a priori} fixed order. The affine constraints in SLP~\cite{wang2019system}, IOP~\cite{furieri2019input} (see \cref{eq:iop}) make their numerical computation non-trivial. The FIR approximation in~\cite{wang2019system,furieri2019input} often leads to controllers of very high order that are impractical to deploy. Furthermore, the FIR approximation may make SLP~\cite{wang2019system} infeasible even for very simple systems; see~\cite{zheng2019systemlevel}. In contrast, the single affine constraint in \Cref{theorem:parameterization} allows for a robust filtering interpretation~\cref{eq:affine_residual} and admits an efficient LMI~\cref{eq:stabilization_LMI}~for~all stabilizing controllers. Moreover, the controller from~\cref{eq:stabilization_LMI} always has the same order as the system state in~\cref{eq:LTI}. \textit{To the best our knowledge, \Cref{theorem:feas_lmi} offers the first efficient LMI among the recent surging interest on frequency-domain characterizations of $\mathcal{C}_{\mathrm{stab}}$~\cite{wang2019system,furieri2019input,zheng2019systemlevel,tseng2021realization,zheng2019equivalence}.}  \hfill $\square$
\end{remark}

\begin{remark}[Comparison with standard LMI for stability]
For internal stability, one can also derive an LMI based on \Cref{lemma:internalstability}. In particular, the following bilinear inequality
   \begingroup
    \setlength\arraycolsep{2pt}
\def\arraystretch{0.9}
\begin{equation} \label{eq:LyapunovLMI}
   \small \begin{bmatrix}
    A+BD_\mK C & BC_\mK\\
  B_\mK C & A_\mK
    \end{bmatrix} P + P\begin{bmatrix}
    A+BD_\mK C & BC_\mK\\
  B_\mK C & A_\mK
    \end{bmatrix}^\tr \prec 0, 
\end{equation}
\endgroup
with $P \succ 0$ can be linearized into an LMI using a nonlinear change of variables~\cite{gahinet1994linear}; see also~\cite[Section 3]{zheng2021analysis} for a recent~revisit. However, the change of variables for~\cref{eq:LyapunovLMI} has a complicated inverse and factorization. Our new controller in~\cref{eq:X_Y_realization_stabilization} is more straightforward (with only inverse on diagonal blocks), which offers benefits in other scenarios, e.g., decentralized control~\cite{rotkowitz2005characterization,furieri2019sparsity}. 
\hfill $\square$
\end{remark}

\subsection{Decentralized stabilization} \label{subsection:decentralized}
 
 One main motivation for the recent surging interests in frequency-domain characterizations of $\mathcal{C}_{\mathrm{stab}}$~\cite{wang2019system,furieri2019input,zheng2019systemlevel,tseng2021realization,zheng2019equivalence} is that one can impose structural constraints on the design parameters that can lead to structural controller constraints, such as a decentralized controller $\mathbf{K}$. Note that imposing convex constraints on $\mathbf{K}$ often leads to intractable synthesis problems~\cite{rotkowitz2005characterization,furieri2019sparsity}, while imposing convex constraints on the new parameters after reparameterization of $\mathcal{C}_{\mathrm{stab}}$ naturally leads to a  convex (but infinitely-dimensional) problem; see, e.g.,~\cite[Section IV]{zheng2019equivalence}.  

Research in decentralized control has remains of great  interest~\cite{bakule2008decentralized}, especially for large-scale interconnected systems.  This aims to design a decentralized  controller based on local measurements for each subsystem to regulate the global behavior. 
In our LMI computation~\cref{eq:stabilization_LMI}, structural constraints on $\mathbf{X}(z)$ and $\mathbf{Y}(z)$ may be enforced by constraining the decision~variables $Z$, $Q$, $F$, $L_X$, $L_Y$,
  $R_X$, and $R_Y$. In particular, if all these variables have a block-diagonal (decentralized) structure
  then $\mathbf{X}(z)$ and $\mathbf{Y}(z)$ also have the same block-diagonal
  structure, hence $\mathbf{K}(z) = \mathbf{Y}(z) \mathbf{X}^{-1}(z)$ will be block diagonal (decentralized), so is the state-space realization in~\cref{eq:state-space-K}. Note that imposing a block diagonal constraint on the Youla parameter $\mathbf{Q}$ does not lead to a decentralized controller $\mathbf{K}$ (see~\cite{rotkowitz2005characterization,furieri2019sparsity} more discussions on constraints for $\mathbf{Q}$).  %Note that this is in contrast to Youla parameterization 

\section{Numerical Experiments} \label{section:experiments}

In this section, we consider a discrete-time LTI system that consists of $n$ subsystems interacting over a chain graph (see \Cref{figure:LineGraph}) to illustrate the performance of our LMI-based computation in \Cref{theorem:feas_lmi} and \Cref{corollary:stability}. We used YALMIP~\cite{lofberg2004yalmip} together with the solver MOSEK~\cite{andersen2000mosek} to solve the optimization problems in our numerical experiments.

\subsection{Example setup}
Similar to~\cite{zheng2017scalable},  
we assume the dynamics of each node $x_i$ are an unstable second-order system coupled with its neighbouring nodes through an exponentially decaying function~as
%\begin{equation}
\begin{align}
    x_i[t+1] &= \begin{bmatrix}
    1 & 1 \\ -1 & 2
    \end{bmatrix}x_i[t] + \sum_{j \in \mathcal{N}_i} \alpha(i,j) x_j[t] + \begin{bmatrix}
    0\\1
    \end{bmatrix} u_i[t], \nonumber\\
    y_i[t] & = \begin{bmatrix} 0 & 1 \end{bmatrix}x_i[t], \label{eq:chain-dynamics}
\end{align}
%\end{equation}
where $\alpha(i,j) = \frac{1}{5}e^{-(i-j)^2}$, $\mathcal{N}_i = \{i-1, i+1\} \cap \{1, \ldots, n\}$ and $i = 1, \ldots, n$. Our goal is to design a decentralized dynamical controller for each subsystem $i$ based on its own measurement $\mathbf{u}_i = \mathbf{K}_i \mathbf{y}_i$ to stabilize the global system. 

We first get a doubly coprime factorization of this system by the standard pole placement method in which the closed-loop poles were chosen randomly from $-0.5$ to $0.5$ (see Appendix A for the computation of a doubly coprime factorization). As discussed in \Cref{subsection:decentralized}, we can constrain the decision variables $Z$, $Q$, $F$, $L_X$, $L_Y$,  $R_X$, and $R_Y$ to be block diagonal with dimensions consistent with each subsystem. This leads to block diagonal $\mathbf{X}(z)$ and $\mathbf{Y}(z)$, and thus results in a desired decentralized controller. In particular, we solved the following optimization problem\footnote{See our code at { \url{https://github.com/soc-ucsd/iop_lmi}}.}
  \begin{equation} \label{eq:lmi-decentralized}
      \begin{aligned}
      \min \quad &  h(Q,F, L_X,L_Y, R_X, R_Y)  \\
      \mathrm{subject~to} \quad & \cref{eq:stabilization_LMI}, \\
       & X \succ 0,\; \; Z, Q, F \; \text{block diagonal}, \\
      & L_X, L_Y, R_X, R_Y \; \text{block diagonal},
      \end{aligned}
  \end{equation}
where we chose $\epsilon = 1$ in \cref{eq:stabilization_LMI} to guarantee stability and the cost function $h(R,F, L_X,L_Y, R_X, R_Y) :=  \|Q\|_\infty + \|F\|_\infty+ \|L_X\|_\infty+ \|L_Y\|_\infty+ \|R_X\|_\infty+ \|R_Y\|_\infty$ with $\|V\|_\infty := \max_{ij}|V_{ij}|$ is to regularize the size of the controller realization. For the comparison of numerical efficiency, we also solved a centralized $\mathcal{H}_2$ optimal control problem using the SLP~\cite{wang2019system} according to the setup in~\cite[Section 7]{zheng2019systemlevel}, where the standard FIR approximation was used in numerical computation\footnote{We used the implementations of closed-loop parameterizations \cite{zheng2019systemlevel} (including SLP) at \url{https://github.com/soc-ucsd/h2_clp}.}.

\subsection{Numerical results and computational efficiency}

    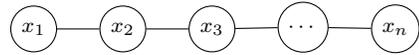
\begin{figure}[t]
	\centering
	\footnotesize
	\begin{tikzpicture}
	\matrix (m) [matrix of nodes,row sep = 0em,	column sep = 2em, nodes={circle, draw=black}] at (0,0) {$x_1$ & $x_2$ & $x_3$ & $\cdots$  & $x_n$\\};
	\draw (m-1-1) edge (m-1-2);
		\draw (m-1-2) edge (m-1-3);
			\draw (m-1-3) edge (m-1-4);
				\draw (m-1-4) edge (m-1-5);
	\end{tikzpicture}
	\caption{A dynamical system interacting over a chain graph, where the dynamics of each subsystem $x_i$ only depend on its neighbors $x_{i-1}, x_{i+1}$. }
	\label{figure:LineGraph}
\end{figure}

\begin{figure}[t]
\setlength{\belowcaptionskip}{0pt}
\setlength{\abovecaptionskip}{0pt}
    \centering
     \subfigure[]{\includegraphics[scale = 0.6]{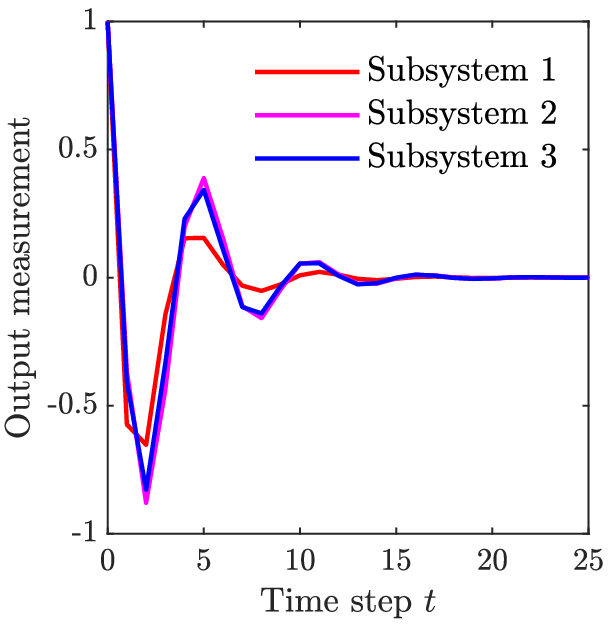}} \hspace{5mm}
     \subfigure[]{\includegraphics[scale = 0.6]{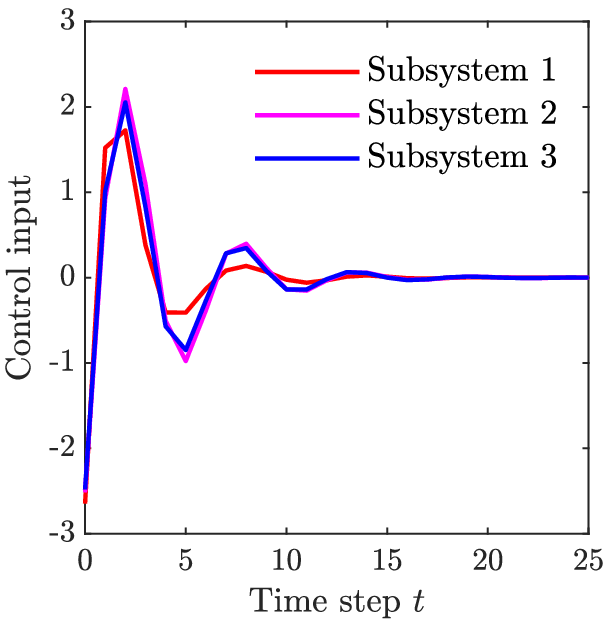}}
    \caption{Responses of~\cref{eq:chain-dynamics} with three subsystems $n = 3$. The decentralized controller $\mathbf{u}_i = \mathbf{K}_i\mathbf{y}_i$, shown in~\cref{eq:decentralized-controller-lmi}, was computed via solving~\cref{eq:lmi-decentralized}. (a) Output measurement $y_i[t]$; (b) Input $u_i[t]$. }
    \label{fig:response-lmi}
\end{figure}
We first consider an LTI system~\cref{eq:chain-dynamics} with $n = 3$ subsystems. For this small system, it took less than half a second to solve~\cref{eq:lmi-decentralized}, resulting in the following decentralized controller
\begin{equation} \label{eq:decentralized-controller-lmi}
    \begin{aligned}
    \mathbf{K}_1 &= \frac{-2.647 z^2 - 0.04603 z - 0.02581}{z^2 + 0.01875 z + 0.009845},\\
    \mathbf{K}_2 &= \frac{ -2.515 z^2 + 0.1379 z - 0.09867}{z^2 - 0.05334 z + 0.03937},\\
    \mathbf{K}_3 &= \frac{ -2.481 z^2 + 0.1326 z - 0.06773}{z^2 - 0.05207 z + 0.02741}.
    \end{aligned}
\end{equation}
The order of each local controller $\mathbf{u}_i = \mathbf{K}_i\mathbf{y}_i$ is guaranteed to be the same with the state dimension of each subsystem (which is two in this case). \Cref{fig:response-lmi} shows the the responses of input $u_i[t]$ and output $y_i[t]$  when the initial state was $x_i[0] = \begin{bmatrix} 0,\; 1 \end{bmatrix}^\tr, i = 1, 2, 3$. As expected, the decentralized controller from~\cref{eq:lmi-decentralized} stabilizes the global system~\cref{eq:chain-dynamics}.  

For comparison, we computed a centralized $\mathcal{H}_2$ optimal controller via SLP~\cite{wang2019system} according to~\cite[Section 7]{zheng2019systemlevel}. This~SLP problem is infinite dimensional, and we used a standard FIR approximation for computation. \Cref{fig:response-sls} (a) and (b) demonstrate the closed-loop responses using the resulting centralized dynamic controller when the FIR length was $10$ and $20$, respectively. Note that the FIR approximation always leads to a dynamical controller of high order (which scales linearly with respect to the FIR length): in particular, the order of the controller with FIR length $10$ is $84$ and the order of  the controller with FIR length $20$  is $174$. In contrast, our LMI-based computation in \Cref{theorem:feas_lmi} and \Cref{corollary:stability} guarantees that the order of the resulting controller will be the same as the order of the system. %The order of each decentralized controller in~\cref{eq:decentralized-controller-lmi} is two.  % that is two in this case. 

Moreover, it is known that the computational efficiency of the FIR approximation does not scale well with system dimension, as it leads to optimization problems of very large size. To illustrate this, we varied the number of subsystems from 6 to 14 in~\cref{eq:chain-dynamics} and allow  each subsystem to use its own state (i.e., $y_i[t] = x_i[t]$). \Cref{tab:time-consumption} lists the time consumption for solving~\cref{eq:lmi-decentralized} and the SLP problem with FIR length $20$. It is clear that our LMI-based computation is much more scalable. For the case $n = 14$, our LMI was two orders of magnitude faster to solve. Finally, as shown in \Cref{tab:order}, the order of the dynamic controller \cref{eq:lmi-decentralized} is always two whereas the order of the controller from the SLP increases dramatically, and is of order $1092$ when $n = 14$, which is unpractical for deployment.

%{\color{red} a short paragraph }

\begin{figure}[t]
\setlength{\belowcaptionskip}{0pt}
\setlength{\abovecaptionskip}{3pt}
    \centering
     %\subfigure[]
     {\includegraphics[scale = 0.6]{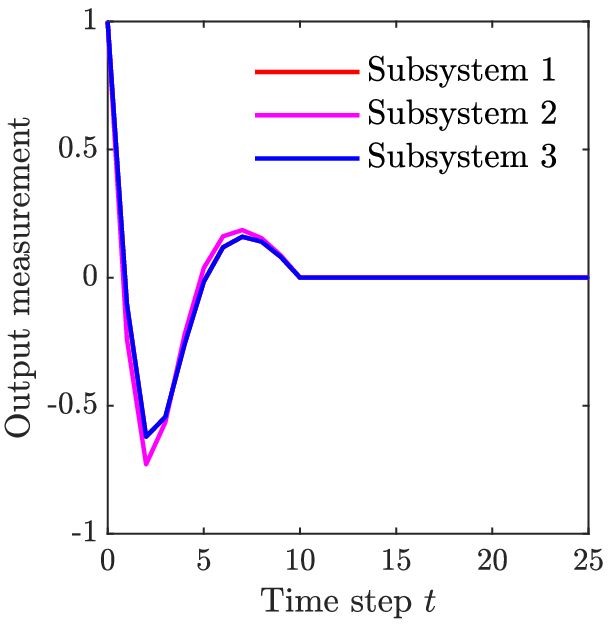}} \hspace{5mm}
     %\subfigure[]
     {\includegraphics[scale = 0.6]{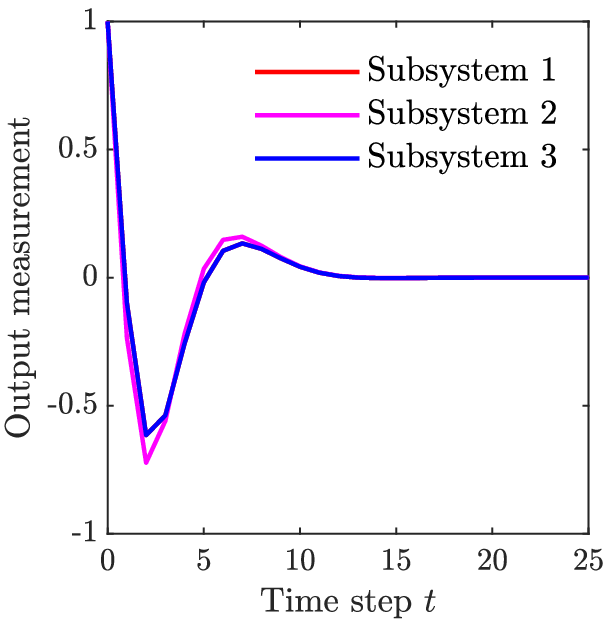}}
    \caption{Responses of~\cref{eq:chain-dynamics} with three subsystems $n = 3$ under a centralized $\mathcal{H}_2$ optimal controller computed via SLP and FIR approximation. (Left) FIR length equals to $10$; (Right) FIR length equals to $20$. }
    \label{fig:response-sls}
\end{figure}

% \begin{table}[t]
% %\setlength{\belowcaptionskip}{0pt}
% %\setlength{\abovecaptionskip}{0pt}
% \caption{Time (in seconds), including YALMIP time and MOSEK time) and controller order %for computing dynamical controllers 
% using~\cref{eq:lmi-decentralized} and SLP + FIR approximation (length 20). }\label{tab:time-consumption}
% {\small
% 	\centering
% 	\begin{tabular}{ccrrrrr}
% 	\toprule
% 	  \multicolumn{2}{c}{Number of nodes $n$ } & 6  &  8 &   10&    12 & 14\\ % &  100 \\
% 		\midrule
% 		%\hline
% \multirow{2}{*}{LMI~\cref{eq:lmi-decentralized}} & Time &  0.49 &   0.60 &    0.75 &    0.99 &    1.28
%     \\
%     	 & Order &  2 &   2 &    2 &    2 &   2
%     \\
% 	\multirow{2}{*}{	SLP + FIR }& Time &    3.22   & 8.60 &  22.68 &   53.19&  132.87
%  \\
%   & Order & 468 & 624  &   780  & 936   &    1092
%     \\
% 		%SLP + FIR $T=40$ & 12.20 & 15.60 &  ** & ** & ** \\
% 	\bottomrule
% 	\end{tabular}
% 	}
% \end{table}

\begin{table}[t!]
\setlength{\belowcaptionskip}{0pt}
\setlength{\abovecaptionskip}{0pt}
\centering
\caption{Time (in seconds) 
%for computing dynamical controllers 
for~\cref{eq:lmi-decentralized} and SLP + FIR (length 20). Includes YALMIP time and MOSEK time.}
\label{tab:time-consumption}
{\small
	\centering
	\begin{tabular}{crrrrr}
	\toprule
	  \# of nodes $n$ & 6  &  8 &   10&    12 & 14\\ % &  100 \\
		\midrule
		%\hline
     LMI~\cref{eq:lmi-decentralized} &  0.49 &   0.60 &    0.75 &    0.99 &    1.28
    \\
	SLP + FIR&    3.22   & 8.60 &  22.68 &   53.19&  132.87
    \\
		%SLP + FIR $T=40$ & 12.20 & 15.60 &  ** & ** & ** \\
	\bottomrule
	\end{tabular}
	}
\end{table}

\begin{table}[t!]
\setlength{\belowcaptionskip}{0pt}
\setlength{\abovecaptionskip}{0pt}
\centering
\caption{Controller order %for computing dynamical controllers 
for~\cref{eq:lmi-decentralized} and SLP + FIR (length 20). }
\label{tab:order}
{\small
	\centering
	\begin{tabular}{crrrrr}
	\toprule
	  \# of nodes $n$ & 6  &  8 &   10&    12 & 14\\ % &  100 \\
		\midrule
		%\hline
      LMI~\cref{eq:lmi-decentralized} &  2 &   2 &    2 &    2 &   2
    \\
	  SLP + FIR & 468 & 624  &   780  & 936   &    1092
    \\
		%SLP + FIR $T=40$ & 12.20 & 15.60 &  ** & ** & ** \\
	\bottomrule
	\end{tabular}
	}
\end{table}

% {\color{red}
% Some simulation

% \subsection{Example 1: Comparison with SLP/IOP}

% Academic example for stabilization;  and comparison with FIR-based solutions - SLP and IOP- time consumption, optimization problem size, the order of the final controllers?

% \subsection{Example 2: Mixed traffic control}

% Demonstrate decentralized stabilization

% %consider a simplified LCC~\cite{wang2021leading}, similar to~\cite{zheng2019system} but with more vehicles. 
% }

\balance
\section{Conclusions} \label{section:conclusion}

In this paper, we have presented a kernel version of the Youla parameterization for stabilizing controllers $\mathcal{C}_{\mathrm{stab}}$. This parameterization only involves a single affine constraint, which can be viewed as a novel robust filtering problem. This filtering perspective leads to the first efficient LMI characterization for the frequency-domain characterization of $\mathcal{C}_{\mathrm{stab}}$. Our LMI characterization offers significant advantages compared to the existing parameterizations (SLP~\cite{wang2019system}, IOP~\cite{furieri2019input}, and the mixed versions~\cite{zheng2019systemlevel}) in terms of both computation and implementation. Ongoing research directions include investigations on LMIs for performance specifications under our new controller parameterization.

\bibliographystyle{IEEEtran}
\bibliography{references}

\begin{thebibliography}{10}
\providecommand{\url}[1]{#1}
\csname url@samestyle\endcsname
\providecommand{\newblock}{\relax}
\providecommand{\bibinfo}[2]{#2}
\providecommand{\BIBentrySTDinterwordspacing}{\spaceskip=0pt\relax}
\providecommand{\BIBentryALTinterwordstretchfactor}{4}
\providecommand{\BIBentryALTinterwordspacing}{\spaceskip=\fontdimen2\font plus
\BIBentryALTinterwordstretchfactor\fontdimen3\font minus
  \fontdimen4\font\relax}
\providecommand{\BIBforeignlanguage}[2]{{%
\expandafter\ifx\csname l@#1\endcsname\relax
\typeout{** WARNING: IEEEtran.bst: No hyphenation pattern has been}%
\typeout{** loaded for the language `#1'. Using the pattern for}%
\typeout{** the default language instead.}%
\else
\language=\csname l@#1\endcsname
\fi
#2}}
\providecommand{\BIBdecl}{\relax}
\BIBdecl

\bibitem{zhou1996robust}
K.~Zhou, J.~C. Doyle, K.~Glover \emph{et~al.}, \emph{Robust and optimal
  control}.\hskip 1em plus 0.5em minus 0.4em\relax Prentice hall New Jersey,
  1996, vol.~40.

\bibitem{boyd1991linear}
S.~P. Boyd and C.~H. Barratt, \emph{Linear controller design: limits of
  performance}.\hskip 1em plus 0.5em minus 0.4em\relax Prentice Hall Englewood
  Cliffs, NJ, 1991.

\bibitem{francis1987course}
A.~Francis, \emph{A course in $\mathcal{H}_\infty$ control theory}.\hskip 1em
  plus 0.5em minus 0.4em\relax Springer-Verlag, 1987.

\bibitem{youla1976modern}
D.~Youla, H.~Jabr, and J.~Bongiorno, ``Modern {Wiener-Hopf} design of optimal
  controllers--{Part II}: The multivariable case,'' \emph{IEEE Trans. Autom.
  Control.}, vol.~21, no.~3, pp. 319--338, 1976.

\bibitem{anderson2019system}
J.~Anderson, J.~C. Doyle, S.~H. Low, and N.~Matni, ``System level synthesis,''
  \emph{Annual Reviews in Control}, 2019.

\bibitem{furieri2019input}
L.~Furieri, Y.~Zheng, A.~Papachristodoulou, and M.~Kamgarpour, ``An
  input-output parametrization of stabilizing controllers: amidst {Youla} and
  system level synthesis,'' \emph{IEEE Control Systems Letters}, vol.~3, no.~4,
  pp. 1014--1019, Oct 2019.

\bibitem{zheng2021sample}
Y.~Zheng, L.~Furieri, M.~Kamgarpour, and N.~Li, ``Sample complexity of linear
  quadratic gaussian (lqg) control for output feedback systems,'' in
  \emph{Learning for Dynamics and Control}.\hskip 1em plus 0.5em minus
  0.4em\relax PMLR, 2021, pp. 559--570.

\bibitem{dean2020sample}
S.~Dean, H.~Mania, N.~Matni, B.~Recht, and S.~Tu, ``On the sample complexity of
  the linear quadratic regulator,'' \emph{Foundations of Computational
  Mathematics}, vol.~20, no.~4, pp. 633--679, 2020.

\bibitem{zhang2021sample}
Y.~Zhang, S.~K. Ukil, E.~Neimand, S.~Sabau, and M.~E. Hohil, ``Sample
  complexity of the robust {LQG} regulator with coprime factors uncertainty,''
  \emph{arXiv preprint arXiv:2109.14164}, 2021.

\bibitem{zheng2019equivalence}
Y.~Zheng, L.~Furieri, A.~Papachristodoulou, N.~Li, and M.~Kamgarpour, ``On the
  equivalence of youla, system-level, and input--output parameterizations,''
  \emph{IEEE Transactions on Automatic Control}, vol.~66, no.~1, pp. 413--420,
  2020.

\bibitem{tseng2021realization}
S.-H. Tseng, ``Realization-stability lemma for controller synthesis,''
  \emph{arXiv preprint arXiv:2112.02005}, 2021.

\bibitem{zheng2019systemlevel}
Y.~Zheng, L.~Furieri, M.~Kamgarpour, and N.~Li, ``System-level, input--output
  and new parameterizations of stabilizing controllers, and their numerical
  computation,'' \emph{Automatica}, vol. 140, p. 110211, 2022.

\bibitem{wang2019system}
Y.-S. Wang, N.~Matni, and J.~C. Doyle, ``A system level approach to controller
  synthesis,'' \emph{IEEE Trans. Autom. Control.}, 2019.

\bibitem{furieri2019sparsity}
L.~Furieri, Y.~Zheng, A.~Papachristodoulou, and M.~Kamgarpour, ``Sparsity
  invariance for convex design of distributed controllers,'' \emph{IEEE Trans.
  Control Netw. Syst.}, pp. 1--12, 2020.

\bibitem{rotkowitz2005characterization}
M.~Rotkowitz and S.~Lall, ``A characterization of convex problems in
  decentralized control,'' \emph{IEEE transactions on Automatic Control},
  vol.~50, no.~12, pp. 1984--1996, 2005.

\bibitem{geromel2000h}
J.~C. Geromel, J.~Bernussou, G.~Garcia, and M.~C. de~Oliveira,
  ``$\mathcal{H}_2$ and $\mathcal{H}_\infty$ robust filtering for discrete-time
  linear systems,'' \emph{SIAM Journal on Control and Optimization}, vol.~38,
  pp. 1353--1368, 2000.

\bibitem{geromel2002robust}
J.~C. Geromel, M.~C. de~Oliveira, and J.~Bernussou, ``Robust filtering of
  discrete-time linear systems with parameter dependent lyapunov functions,''
  \emph{SIAM Journal on control and optimization}, vol.~41, no.~3, pp.
  700--711, 2002.

\bibitem{dullerud2013course}
G.~E. Dullerud and F.~Paganini, \emph{A course in robust control theory: a
  convex approach}.\hskip 1em plus 0.5em minus 0.4em\relax Springer Science \&
  Business Media, 2013, vol.~36.

\bibitem{gahinet1994linear}
P.~Gahinet and P.~Apkarian, ``A linear matrix inequality approach to
  $\mathcal{H}_\infty$ control,'' \emph{International Journal of Robust and
  Nonlinear Control}, vol.~4, no.~4, pp. 421--448, 1994.

\bibitem{scherer1997multiobjective}
C.~Scherer, P.~Gahinet, and M.~Chilali, ``Multiobjective output-feedback
  control via {LMI} optimization,'' \emph{IEEE Transactions on Automatic
  Control}, vol.~42, no.~7, pp. 896--911, 1997.

\bibitem{zheng2021analysis}
Y.~Zheng, Y.~Tang, and N.~Li, ``Analysis of the optimization landscape of
  {Linear Quadratic Gaussian (LQG)} control,'' \emph{arXiv preprint
  arXiv:2102.04393}, 2021.

\bibitem{nett1984connection}
C.~Nett, C.~Jacobson, and M.~Balas, ``A connection between state-space and
  doubly coprime fractional representations,'' \emph{IEEE Trans. Autom.
  Control.}, vol.~29, no.~9, pp. 831--832, 1984.

\bibitem{de2002extended}
M.~C. De~Oliveira, J.~C. Geromel, and J.~Bernussou, ``Extended $\mathcal{H}_2$
  and $\mathcal{H}_\infty$ norm characterizations and controller
  parametrizations for discrete-time systems,'' \emph{International Journal of
  Control}, vol.~75, no.~9, pp. 666--679, 2002.

\bibitem{bakule2008decentralized}
L.~Bakule, ``Decentralized control: An overview,'' \emph{Annual reviews in
  control}, vol.~32, no.~1, pp. 87--98, 2008.

\bibitem{lofberg2004yalmip}
J.~L{\"o}fberg, ``Yalmip: A toolbox for modeling and optimization in matlab,''
  in \emph{Proceedings of the CACSD Conference}, vol.~3.\hskip 1em plus 0.5em
  minus 0.4em\relax Taipei, Taiwan, 2004.

\bibitem{andersen2000mosek}
E.~D. Andersen and K.~D. Andersen, ``The {MOSEK} interior point optimizer for
  linear programming: an implementation of the homogeneous algorithm,'' in
  \emph{High performance optimization}.\hskip 1em plus 0.5em minus 0.4em\relax
  Springer, 2000.

\bibitem{zheng2017scalable}
Y.~Zheng, R.~P. Mason, and A.~Papachristodoulou, ``Scalable design of
  structured controllers using chordal decomposition,'' \emph{IEEE Transactions
  on Automatic Control}, vol.~63, no.~3, pp. 752--767, 2017.

\end{thebibliography}

\onecolumn
\section*{Appendix}

\subsection{State-space realization of the coprime factorization} \label{app:coprime-factorization}

It is straightforward to find a doubly coprime factorization for $\mathbf{G}(z)$ given a stabilizable and detectable state-space realization~\cite[Theorem 5.9]{zhou1996robust}. This amounts to find a stabilizing feedback gain and observer gain. 

\vspace{1mm}

    \begin{theorem}
    Suppose $\mathbf{G}(s)$ is a proper real-rational matrix and 
    $
        \mathbf{G} = \left[ \begin{array}{c|c} A & B \\
        \hline
                                     C & D \end{array} \right]
    $
    is a stabilizable and detectable state-space realization. Let $F$ and $L$ be such that $A + BF$ and $A + LC$ are both stable. Then, a doubly co-prime factorization of $\mathbf{G}$ is
    \begin{equation} \label{eq:state-space-coprime}
        \begin{aligned}
            \begin{bmatrix} \mathbf{M}_r & \mathbf{V}_r \\ \mathbf{N}_r & \mathbf{U}_r\end{bmatrix} = \left[ \begin{array}{c|cc} A+BF & B & L \\
        \hline
                                     F & I & 0\\
                                     C+DF & D & I\end{array} \right], \quad 
            \begin{bmatrix} \mathbf{U}_l & -\mathbf{V}_l \\ -\mathbf{N}_l & \mathbf{M}_l\end{bmatrix}  = \left[ \begin{array}{c|cc} A+LC & -(B+LD) & L \\
        \hline
                                     F & I & 0\\
                                     C & -D & I\end{array} \right]. 
        \end{aligned}
    \end{equation}
    \end{theorem}
    \vspace{2mm}

    We can directly verify that the choices in~\cref{eq:state-space-coprime} satisfy \Cref{definition:coprime} (see~\cite{nett1984connection} for detailed computations). The coprime factorization of a transfer matrix in~\cref{eq:state-space-coprime} has a feedback control interpretation~\cite[Remark 5.3]{zhou1996robust}. For example, the right coprime factorization comes out naturally from changing the control variable by a state feedback. 
    
    Consider the state-space model
$$
    \begin{aligned}
        x[t+1] &= Ax[t] + Bu[t], \\
        y[t] & = C x[t] + Du[t].
    \end{aligned}
$$
Introduce a state feedback and change the variable
$
    v[t]:= u[t] - Fx[t]
$
where $F$ makes $A + BF$ stable. We then get 
$$
\begin{aligned}
        x[t+1] &= (A + BF)x[t] + Bv[t], \\
        u[t] &= Fx[t] + v[t] \\
        y[t] &= (C+DF)x[t] + Dv[t].
\end{aligned}
$$
From these equations, it is easy to see that the transfer matrix from $\mathbf{v}$ to $\mathbf{u}$ is 
 $$
        \mathbf{M}_r(z) = \left[ \begin{array}{c|c} A+BF & B \\
        \hline
                                     F & I \end{array} \right],
    $$
and that the transfer matrix from $\mathbf{v}$ to $\mathbf{y}$ is 
$$
        \mathbf{N}_r(z) = \left[ \begin{array}{c|c} A+BF & B \\
        \hline
                                     C+DF & D \end{array} \right].
    $$
Therefore, we have 
$
    \mathbf{u} = \mathbf{M}_r\mathbf{v}, \, \mathbf{y} = \mathbf{N}_r\mathbf{v},
$
so that $\mathbf{y} = \mathbf{N}_r\mathbf{M}_r^{-1}\mathbf{u}$, \emph{i.e.}, $\mathbf{G} = \mathbf{N}_r\mathbf{M}_r^{-1}$. %Similar interpretation on the left coprime factors  $\mathbf{N}_l, \mathbf{M}_l$ can be given by a stable observer gain $L$; see.

\subsection{Computation of~\cref{eq:transformation}}  \label{app:auxiliary_computation_primal}
Here, we provide some detailed computation for~\cref{eq:transformation}.  For notational convenience, we highlight $ \blue{\hat{A}}, \blue{\hat{B}}, \blue{\hat{C}}, \blue{\hat{D}}$ in blue.  
\begin{itemize}
    \item For~\cref{eq:transformation_a}, we can verify 
\begin{equation*} %\label{eq:computation_a}
    \begin{bmatrix}
    I & 0 \\ 0 & -V^\tr N
    \end{bmatrix}^\tr  \begin{bmatrix} X & U \\ U^\tr & \hat{X} \end{bmatrix}  \begin{bmatrix}
    I & 0 \\ 0 & -V^\tr N
    \end{bmatrix} = \begin{bmatrix}
    I & 0 \\ 0 & -NV
    \end{bmatrix}\begin{bmatrix} X & -UV^\tr N \\ U^\tr & -\hat{X}V^\tr N \end{bmatrix} = \begin{bmatrix} X & -UV^\tr N \\ -NVU^\tr & NV\hat{X}V^\tr N \end{bmatrix} 
\end{equation*}
Since $\tilde{P}\tilde{P}^{-1} = I$, we have 
\begin{equation} \label{eq:Uinverse}
    U^\tr Y + \hat{X}V^\tr = 0 \quad \Rightarrow \quad -U^\tr = \hat{X}V^\tr N.
\end{equation}
Combining $Z = -UV^\tr N$ with the two equations above leads to~\cref{eq:transformation_a}.  

\item For~\cref{eq:transformation_b}, we have 
\begin{equation} \label{eq:computation_b}
\begin{aligned}
    \tilde{T}^\tr \tilde{A} \tilde{P} \tilde{T} &=  \begin{bmatrix}
    I & 0 \\ 0 & -V^\tr N
    \end{bmatrix}^\tr \begin{bmatrix} A & B_1 \blue{\hat{C}} \\
        0 &\blue{\hat{A}} \end{bmatrix}\begin{bmatrix} X & -UV^\tr N \\ U^\tr & -\hat{X}V^\tr N \end{bmatrix} \\
        & = \begin{bmatrix} A & B_1 \blue{\hat{C}} \\
        0 &  - NV\blue{\hat{A}} \end{bmatrix}\begin{bmatrix} X & -UV^\tr N \\ U^\tr & -\hat{X}V^\tr N \end{bmatrix} \\
        &= \begin{bmatrix} AX+B_1 \blue{\hat{C}}U^\tr & -AUV^\tr N -B_1 \blue{\hat{C}}\hat{X}V^\tr N \\ - NV\blue{\hat{A}}U^\tr & NV\blue{\hat{A}}\hat{X}V^\tr N \end{bmatrix}.
\end{aligned}
\end{equation}
Considering~\cref{eq:Uinverse}, we have $NV\blue{\hat{A}}\hat{X}V^\tr N = - NV\blue{\hat{A}}U^\tr $. Also, the change of variables in~\cref{eq:change_of_variables} reads as
$$
    \begin{bmatrix}
      Q & F \\
      L & R
    \end{bmatrix}
     :=
    \begin{bmatrix}
      - N V \blue{\hat{A}}  U^\tr & - N V \blue{\hat{B}} \\
      \blue{\hat{C}} U^\tr & \blue{\hat{D}}
    \end{bmatrix},
$$
Combining~\cref{eq:change_of_variables},~\cref{eq:Uinverse} with~\cref{eq:computation_b} leads to~\cref{eq:transformation_b}.   

\item For~\cref{eq:transformation_c}, we have 
$$
    \begin{bmatrix}
    I & 0 \\ 0 & -V^\tr N
    \end{bmatrix}^\tr \begin{bmatrix} B_1 \blue{\hat{D}} - B_2 \\
       \blue{\hat{B}} \end{bmatrix} = \begin{bmatrix}
    I & 0 \\ 0 & -NV
    \end{bmatrix} \begin{bmatrix} B_1 \blue{\hat{D}} - B_2 \\
       -NV\blue{\hat{B}} \end{bmatrix}.
$$
\item For~\cref{eq:transformation_d}, we have 
$$
    \tilde{C} \tilde{P} \tilde{T} = \begin{bmatrix}  C & D_1 \blue{\hat{C}} \end{bmatrix} \begin{bmatrix} X & -UV^\tr N \\ U^\tr & -\hat{X}V^\tr N \end{bmatrix}  = \begin{bmatrix}
    CX + D_1 \blue{\hat{C}}U^\tr & -CUV^\tr N - D_1 \blue{\hat{C}} \hat{X}V^\tr N
    \end{bmatrix}
$$
Combining $Z = -UV^\tr N$ and~\cref{eq:Uinverse} with the equation above leads to~\cref{eq:transformation_d}.   
\end{itemize}

\subsection{Proof of~\Cref{corollary:stability}}

The first part of \Cref{corollary:stability} is immediate. Given $\mathbf{X}$ and $\mathbf{Y}$ in~\cref{eq:X_Y_realization_stabilization}, we prove the following state space realization  
   \begin{equation} \label{eq:state-space-K-app}
        \mathbf{K} = \mathbf{Y}\mathbf{X}^{-1} = \left [
      %\hspace{-.2ex}
      \begin{array}{c|c}
        \hat{A} - \hat{B}\hat{D}_X^{-1}\hat{C}_X & -\hat{B}\hat{D}_X^{-1} \\  \hline 
       -\hat{C}_Y+ \hat{D}_Y \hat{D}_X^{-1}\hat{C}_X & \hat{D}_Y\hat{D}_X^{-1}
      \end{array}
      \hspace{-.2ex}
      \right ].
    \end{equation}
The proof is based on a few standard system operations. We recall some of them below  (see~\cite[Chapter 3.6]{zhou1996robust} for more discussions). Consider two dynamical systems $$
     \mathbf{G}_i = \left[\begin{array}{c|c} A_i & B_i \\ \hline
    C_i & D_i\end{array}\right], \quad i = 1, 2.
$$
Their inverses are given by 
$$
    \mathbf{G}_i^{-1} = \left[\begin{array}{c|c} A_i - B_iD_i^{-1}C_i & -B_iD_i^{-1} \\ \hline
    D_i^{-1}C_i & D_i^{-1}\end{array}\right], i = 1, 2
$$
where we assume $D_i$ is invertible. %If the system is strictly proper, then the inverse will be non-proper and there is no state-space realization. 
The cascade connection of two systems such that $\mathbf{y} = \mathbf{G}_1\mathbf{G}_2\mathbf{u}$ has a state-space realization 
\begin{equation} \label{eq:product}
    \mathbf{G}_1\mathbf{G}_2 = \left[\begin{array}{c c|c} A_1 & B_1C_2 & B_1D_2 \\
    0 & A_2 & B_2 \\ \hline
    C_1 & D_1C_2 & D_1D_2 \end{array}\right]. 
\end{equation}
Note that~\cref{eq:product} is in general not  minimal (there may be uncontrollable and observable modes). For example, when $\mathbf{G}_1 = \mathbf{G}_2^{-1}$, we have $\mathbf{G}_1\mathbf{G}_2 = I$.   
For any invertible matrix $T$ with compatible dimension,  we have 
\begin{equation} \label{eq:Transformation}
     \mathbf{G}_i = \left[\begin{array}{c|c} A_i & B_i \\ \hline
    C_i & D_i\end{array}\right] =  \left[\begin{array}{c|c} TA_iT^{-1} & TB_i \\ \hline
    C_iT^{-1} & D_i\end{array}\right], i = 1, 2. 
\end{equation}

Now, consider the state-space realization of  $\mathbf{X}$ and $\mathbf{Y}$ in~\cref{eq:X_Y_realization_stabilization}, we have 
$$
    \mathbf{X}^{-1} = \left[\begin{array}{c|c} \hat{A} - \hat{B}\hat{D}_X^{-1}\hat{C}_X & -\hat{B}\hat{D}_X^{-1} \\ \hline
    \hat{D}_X^{-1}\hat{C}_X & \hat{D}_X^{-1}\end{array}\right],
$$
and 
$$
\begin{aligned}
\mathbf{Y}\mathbf{X}^{-1} &= 
    \left[\begin{array}{c|c} \hat{A} & \hat{B} \\ \hline
                           \hat{C}_Y & \hat{D}_Y\end{array}\right]
    \left[\begin{array}{c|c} 
    \hat{A} - \hat{B}\hat{D}_X^{-1}\hat{C}_X & -\hat{B}\hat{D}_X^{-1} \\ \hline
    \hat{D}_X^{-1}\hat{C}_X                  & \hat{D}_X^{-1}\end{array}\right]\\
    &=
    \left[\begin{array}{c c|c} 
    \hat{A} &  \hat{B}\hat{D}_X^{-1}\hat{C}_X & \hat{B}\hat{D}_X^{-1} \\
    0 & \hat{A} - \hat{B}\hat{D}_X^{-1}\hat{C}_X & -\hat{B}\hat{D}_X^{-1}\\ \hline
    \hat{C}_Y & \hat{D}_Y\hat{D}_X^{-1}\hat{C}_X & \hat{D}_Y\hat{D}_X^{-1} \end{array}\right].
\end{aligned}
$$
From~\cref{eq:Transformation}, upon defining a transformation 
\begin{equation*} %\label{eq:tranT}
    T = \begin{bmatrix}
        I & I \\
        0 & I
    \end{bmatrix}, \qquad  T^{-1} = \begin{bmatrix}
        I & -I \\
        0 & I
    \end{bmatrix}
\end{equation*}
with compatible dimension, we have 
$$
\mathbf{Y}\mathbf{X}^{-1}  =
    \left[\begin{array}{c c|c} 
    \hat{A} &  0 & 0 \\
    0 & \hat{A} - \hat{B}\hat{D}_X^{-1}\hat{C}_X & -\hat{B}\hat{D}_X^{-1}\\ \hline
    \hat{C}_Y & \hat{D}_Y\hat{D}_X^{-1}\hat{C}_X -\hat{C}_Y & \hat{D}_Y\hat{D}_X^{-1} \end{array}\right] =\left [
      %\hspace{-.2ex}
      \begin{array}{c|c}
        \hat{A} - \hat{B}\hat{D}_X^{-1}\hat{C}_X & -\hat{B}\hat{D}_X^{-1} \\  \hline 
       -\hat{C}_Y+ \hat{D}_Y \hat{D}_X^{-1}\hat{C}_X & \hat{D}_Y\hat{D}_X^{-1}
      \end{array}
      \hspace{-.2ex}
      \right ].
$$
This completes the proof of~\cref{eq:state-space-K-app}.

\subsection{Extended LMI formulation}
We use another $\mathcal{H}_\infty$ lemma from~\cite{de2002extended} to derive a new LMI for solving the right robust filtering problem.

\begin{lemma}[{\!\!\cite{de2002extended}}] \label{lemma:Hinf_new}
    Given a stable transfer function $\mathbf{G}(z) = C(zI - A)^{-1}B + D \in \mathcal{RH}_\infty$, then $\|\mathbf{G}(z) \|^2_\infty < \mu$ if and only if there exist a positive definite matrix $P$ and a matrix $G$ such that 
    \begin{equation} \label{eq:HinfLMI-new}
        \begin{bmatrix} P & AG & B & 0 \\ GA^\tr & G+G^\tr - P & 0 & GC^\tr  \\ B^\tr & 0 & I & D^\tr \\ 0 & CG & D & \mu I \end{bmatrix} \succ 0.
    \end{equation}
\end{lemma}

\vspace{1mm}

We refer the interested reader to~\cite{de2002extended} for discussions on the features of this extended LMI \cref{eq:HinfLMI-new} compared to the standard LMI~\cref{eq:HinfLMI}. Using this extended $\mathcal{H}_\infty$ lemma, we can derive anther LMI to solve the right $\mathcal{H}_\infty$ filtering problem~\eqref{eq:hinf_filter}. 

\vspace{1mm}

\begin{theorem} \label{proposition:hinf_filtering_new}
    There exists $\mathbf{F}(z) \in \mathcal{RH}_\infty$ such that~\eqref{eq:hinf_filter} holds if and only if there exist symmetric matrices $E, H$, and matrices $X,Z,N,G$,  and $Q, F, L, R$ such that 
       \begin{align} \label{eq:filter_LMI_new}
      \begin{bmatrix}
        E & G & A X + B_1 L & A (X-N) + B_1 L & B_1 R - B_2 & 0 \\
        \star & H & Q & Q & F & 0 \\
        \star & \star & X + X^\tr - E & X - N + Z^\tr - G & 0 & X^\tr C^\tr + L^\tr D_1^\tr \\
        \star & \star & \star & Z + Z^\tr - H & 0 & (X-N)^\tr C^\tr + L^\tr D_1^\tr  \\
        \star & \star &\star  &\star  &  I & R^\tr D_1^\tr - D_2^\tr \\
        \star & \star & \star & \star & \star & \mu_1 I
      \end{bmatrix} \succ 0.
    \end{align}
    where $\star$ denotes the symmetric counterparts. If~\eqref{eq:filter_LMI_new} holds, a state-space realization of $F(z) = \hat{C}(zI - \hat{A})\hat{B} + \hat{D}$ is
         \begin{align} \label{eq:filter_LMI_state_space_new}
     \begin{bmatrix}
      \hat{A} & \hat{B} \\
      \hat{C} & \hat{D}
    \end{bmatrix} = \begin{bmatrix}
      U Z^{-1} & 0 \\
      0 & I
    \end{bmatrix}\begin{bmatrix}
      Q & F \\
      L & R
    \end{bmatrix}\begin{bmatrix}
      U^{-1} & 0 \\
      0 & I
    \end{bmatrix},% \\
    %Z &:= - N V U^\tr = X - N.
\end{align}
    where $U$ is an arbitrary non-singular matrix. 
\end{theorem}

\end{document}